\documentclass[12pt,reqno]{amsart}
\usepackage{fullpage}

\newtheorem{theorem}{Theorem}[section]
\newtheorem{lemma}[theorem]{Lemma}

\newtheorem{cor}[theorem]{Corollary}
\newtheorem{conj}[theorem]{Conjecture}

\newtheorem{prob}[theorem]{Problem}
\usepackage{graphicx}
\usepackage{color}
\usepackage[dvipsnames]{xcolor}
\usepackage{subfigure}
\usepackage{amssymb}
\usepackage{amsmath,mathrsfs}
\usepackage{colonequals}
\usepackage{hyperref}
\theoremstyle{definition}
\newtheorem{definition}[theorem]{Definition}

\newtheorem{remark}[theorem]{Remark}

\renewcommand{\subset}{\subseteq}

\renewcommand{\epsilon}{\varepsilon}

\newcommand{\abs}[1]{\left|#1\right|}                   % Absolute value notation
\newcommand{\absf}[1]{|#1|}                             % small absolute value signs
\newcommand{\vnorm}[1]{\left\|#1\right\|}    % norm notation
\newcommand{\vnormf}[1]{\|#1\|}                         % norm notation, forced to be small
\newcommand{\vnormt}[1]{\left\|#1\right\|}    % norm notation
\newcommand{\Z}{\mathbb{Z}}                             % Blackboard notation

\newcommand{\E}{\mathbb{E}}

\renewcommand{\d}{\mathrm{d}}

\renewcommand{\P}{\mathbb{P}}
\newcommand{\R}{\mathbb{R}}

\newcommand{\embolden}[1]{\textbf {#1}}
\newcommand{\redA}{\Sigma}
\newcommand{\redb}{\partial^{*}}

\newcommand{\sdimn}{n}
\newcommand{\adimn}{n+1}

\begin{document}

%\title{The Stability of Gaussian Minimal Bubbles}
\title{Hyperstable Sets with Voting and Algorithmic Hardness Applications}

\author{Steven Heilman}
%\author{Alex Tarter}
\address{Department of Mathematics, University of Southern California, Los Angeles, CA 90089-2532}
\email{stevenmheilman@gmail.com}
%\email{atarter@usc.edu}
\date{\today}
\thanks{S. H. is Supported by NSF Grant CCF 1911216}
%60E15, 60G15, 53A10, 58E30
%\subjclass[2010]{60E15, 53A10, 60G15, 58E30}
%\keywords{Gaussian, bubble, minimal surface, calculus of variations}

\begin{abstract}
The noise stability of a Euclidean set $A$ with correlation $\rho$ is the probability that $(X,Y)\in A\times A$, where $X,Y$ are standard Gaussian random vectors with correlation $\rho\in(0,1)$.  It is well-known that a Euclidean set of fixed Gaussian volume that maximizes noise stability must be a half space.

For a partition of Euclidean space into $m>2$ parts each of Gaussian measure $1/m$, it is still unknown what sets maximize the sum of their noise stabilities.  In this work, we classify partitions maximizing noise stability that are also critical points for the derivative of noise stability with respect to $\rho$.  We call a partition satisfying these conditions hyperstable.  Uner the assumption that a maximizing partition is hyperstable, we prove:
\begin{itemize}
\item a (conditional) version of the Plurality is Stablest Conjecture for $3$ or $4$ candidates.
\item a (conditional) sharp Unique Games Hardness result for MAX-m-CUT for $m=3$ or $4$
\item a (conditional) version of the Propeller Conjecture of Khot and Naor for $4$ sets.
\end{itemize}
We also show that a symmetric set that is hyperstable must be star-shaped.

For partitions of Euclidean space into $m>2$ parts of fixed (but perhaps unequal) Gaussian measure, the hyperstable property can only be satisfied when all of the parts have Gaussian measure $1/m$.  So, as our main contribution, we have identified a possible strategy for proving the full Plurality is Stablest Conjecture and the full sharp hardness for MAX-m-CUT: to prove both statements, it suffices to show that sets maximizing noise stability are hyperstable.  This last point is crucial since any proof of the Plurality is Stablest Conjecture must use a property that is special to partitions of sets into equal measures, since the conjecture is false in the unequal measure case.
\end{abstract}

%https://arxiv.org/pdf/1701.01485.pdf
%https://arxiv.org/abs/1708.03808

\maketitle
\setcounter{tocdepth}{1}
\tableofcontents
%
%
% arxiv subjects: math.DG, math.PR, cs.CC?
%
%  MSC:    60E15, 60G15, 58E15, 49Q10, 68Q25
%  60E15 probability, inequalities
%  60G15 probability, gaussian processes
%  53A10, differential geometry, minimal surfaces
%  58e15, global analysis, analysis on manifolds, variational problems concerning extremal problems in several variables
%  49Q05, calculus of variations, minimal surfaces and optimization
%  49Q10, optimization of shapes other than minimal surfaces
%  68Q25  Analysis of algorithms and problem complexity
%
%  keywords: noise stability, social choice theory, calculus of variations, Max-Cut
%
% 35? pages, 4 figures

\section{Introduction}\label{secintro}

The noise stability of a measurable set $\Omega\subset\R^{\adimn}$ with correlation $\rho$ is the probability that $(X,Y)\in A\times A$, where $X,Y$ are standard Gaussian random vectors with correlation $\rho\in(-1,1)$:
\begin{equation}\label{zero0}
\P((X,Y)\in \Omega\times\Omega),
\end{equation}
where $X=(X_{1},\ldots,X_{\adimn}),Y=(Y_{1},\ldots,Y_{\adimn})$, and $\E X_{i}X_{j}=\E Y_{i}Y_{j}=1_{\{i=j\}}$ for all $1\leq i,j\leq \adimn$ and $\E X_{i}Y_{j}=\rho\cdot 1_{\{i=j\}}$ for all $1\leq i,j\leq \adimn$.

The noise stability of a set could also be called its Gaussian heat content.  Inequalities for noise stability have been investigated in many places, including \cite{borell85,ledoux94,mossel12,eldan13,heilman20d,heilman21}.  Besides their intrinsic interest, noise stability inequalities have applications to social choice theory \cite{khot07,mossel10,isaksson11}, the Unique Games Conjecture in theoretical computer science \cite{khot07,mossel10,khot15}, to semidefinite programming algorithms such as MAX-CUT \cite{khot07,isaksson11}, to learning theory \cite{feldman12}, information theory \cite{de17,de18,heilman22}, etc.  For some surveys on these and related topics, see  \cite{odonnell14b,khot10b,heilman20b}.

A basic question, answered in \cite{borell85} is: which measurable Euclidean sets of fixed Gaussian volume maximizes noise stability?  More generally, our primary question of interest will be: which partitions of Euclidean space of fixed Gaussian volumes maximize the sum of their noise stabilities \cite{isaksson11}?  This particular question was asked in \cite{isaksson11}: it is the continuous version of the Plurality is Stablest Conjecture from social choice theory \cite{khot07,mossel10,isaksson11}.  Some ``robust'' statements of this kind have been investigated.  It was shown in \cite{heilman20d,heilman21} that half spaces are the only local maximizers of the noise stability subject to a Gaussian volume constraint.  That is, for any $\Omega\subset\R^{\adimn}$, if for any collection of measurable sets $\{\Omega^{(s)}\}_{s\in(-1,1)}$ with $\Omega^{(0)}=\Omega$ and with constant Gaussian volume $\P(X\in \Omega^{(s)})=\P(X\in\Omega)$, $\forall$ $s\in(-1,1)$ such that
$$\frac{\d}{\d s}\Big|_{s=0}\,\P((X,Y)\in \Omega^{(s)}\times\Omega^{(s)})=0,$$
if we have
$$\frac{\d^{2}}{\d s^{2}}\Big|_{s=0}\,\P((X,Y)\in \Omega^{(s)}\times\Omega^{(s)})\leq0,$$
then $\Omega$ must be a half space (i.e. the set of all points lying on one side of a hyperplane) (up to measure zero changes to $\Omega$).  Borell's inequality follows as a Corollary \cite{borell85}: half spaces maximize noise stability among all Euclidean sets of fixed Gaussian volume.  In fact, \cite{heilman21} also showed that sets close to half spaces nearly maximize noise stability, proving some cases of a Conjecture of Eldan \cite{eldan13}, following similar ``robust'' inequalities of \cite{mossel12,eldan13}.

For partitions of Euclidean space with $m>2$ fixed Gaussian volumes, an analogous characterization of local maxima should hold \cite{isaksson11}, though obtaining such a characterization requires the Gaussian volumes to be equal \cite{heilman14}.  That is, we let $\Omega_{1},\ldots,\Omega_{m}\subset\R^{\adimn}$ be a partition of Euclidean space into $m$ sets such that $\P(X\in \Omega_{i})=1/m$ for all $1\leq i\leq m$.  If for any collection of measurable sets $\{\Omega_{i}^{(s)}\}_{s\in(-1,1),1\leq i\leq m}$ with $\cup_{i=1}^{m}\Omega_{i}^{(s)}=\R^{\adimn}$ for all $s\in(-1,1)$, $\Omega_{i}^{(0)}=\Omega_{i}$ for all $1\leq i\leq m$ and $\P(X\in \Omega_{i}^{(s)})=1/m$ for all $1\leq i\leq m$ such that
$$\frac{\d}{\d s}|_{s=0}\sum_{i=1}^{m}\P((X,Y)\in \Omega_{i}^{(s)}\times\Omega_{i}^{(s)})=0,$$
we have
$$\frac{\d^{2}}{\d s^{2}}|_{s=0}\sum_{i=1}^{m}\P((X,Y)\in \Omega_{i}^{(s)}\times\Omega_{i}^{(s)})\leq0,$$
then $\Omega_{1},\ldots,\Omega_{m}$ should be cones over the regular simplex centered at the origin (up to measure zero changes to the sets).  This conjecture is known as the Standard Simplex Conjecture \cite{isaksson11}, or the continuous version of the Plurality is Stablest Conjecture \cite{khot07}.  When $m=2$, this conjecture has been proven: it is Borell's inequality \cite{borell85} for a set of Gaussian measure $1/2$.  The only known result for $m>2$ is the case $m=3$ and $\rho>0$ small, proven in \cite{heilman20d}.

Since the Gaussian measures we consider have a correlation parameter $\rho\in(-1,1)$, below we write $\P_{\rho}$ in place of $\P$, to denote the dependence of the measure on $\rho$, when appropriate.

\begin{definition}[\embolden{Hyperstable Set}]
We say a measurable set $\Omega\subset\R^{\adimn}$ is \textbf{hyperstable} for correlation parameter $\rho\in(0,1)$ if the following holds.  For any collection of measurable sets $\{\Omega^{(s)}\}_{s\in(-1,1)}$ with constant Gaussian volume $\P(X\in \Omega^{(s)})=\P(X\in\Omega)$ $\forall$ $s\in(-1,1)$, such that $\Omega^{(0)}=\Omega$ and such that
$$\frac{\d}{\d s}|_{s=0}\,\P((X,Y)\in \Omega^{(s)}\times\Omega^{(s)})=0,$$
we have
$$\frac{\d^{2}}{\d s^{2}}|_{s=0}\,\P((X,Y)\in \Omega^{(s)}\times\Omega^{(s)})\leq0,\quad\mathrm{and}\quad
\frac{\d^{2}}{\d s\,\d\rho }|_{s=\rho=0}\,\P_{\rho}((X,Y)\in \Omega^{(s)}\times\Omega^{(s)})=0.$$
%\snote{Require equality of inequality above?}
In the case $\rho\in(-1,0)$, the inequality is reversed.
\end{definition}

\begin{definition}[\embolden{Hyperstable Symmetric Set}]
We say a measurable set $\Omega\subset\R^{\adimn}$ is a \textbf{hyperstable symmetric set} for correlation parameter $\rho\in(0,1)$ if the following holds.  We have $\Omega=-\Omega$ and for any collection of measurable sets $\{\Omega^{(s)}\}_{s\in(-1,1)}$ with constant Gaussian volume $\P(X\in \Omega^{(s)})=\P(X\in\Omega)$ $\forall$ $s\in(-1,1)$ such that $\Omega^{(0)}=\Omega$ and such that
$$\frac{\d}{\d s}|_{s=0}\,\P((X,Y)\in \Omega^{(s)}\times\Omega^{(s)})=0,$$
we have
$$\frac{\d^{2}}{\d s^{2}}|_{s=0}\,\P((X,Y)\in \Omega^{(s)}\times\Omega^{(s)})\leq0,\quad\mathrm{and}\quad\frac{\d^{2}}{\d s\,\d\rho }|_{s=\rho=0}\,\P_{\rho}((X,Y)\in \Omega^{(s)}\times\Omega^{(s)})=0.$$
%\snote{Require equality of inequality above?}
In the case $\rho\in(-1,0)$, the inequality is reversed.
\end{definition}

\begin{definition}[\embolden{Hyperstable Partition}]\label{hyppart}
We say that a partition of Euclidean space into measurable sets $\Omega_{1},\ldots,\Omega_{m}\subset\R^{\adimn}$ is \textbf{hyperstable} for correlation parameter $\rho\in(0,1)$ if the following holds.  For any collection of measurable sets $\{\Omega_{i}^{(s)}\}_{s\in(-1,1),1\leq i\leq m}$ with $\cup_{i=1}^{m}\Omega_{i}^{(s)}=\R^{\adimn}$ for all $s\in(-1,1)$, $\Omega_{i}^{(0)}=\Omega_{i}$ for all $1\leq i\leq m$ and $\P(X\in \Omega_{i}^{(s)})=1/m$ for all $1\leq i\leq m$ such that
$$\frac{\d}{\d s}|_{s=0}\sum_{i=1}^{m}\P((X,Y)\in \Omega_{i}^{(s)}\times\Omega_{i}^{(s)})=0,$$
we have
$$\frac{\d^{2}}{\d s^{2}}|_{s=0}\sum_{i=1}^{m}\P((X,Y)\in \Omega_{i}^{(s)}\times\Omega_{i}^{(s)})\leq0,\quad\mathrm{and}\quad\frac{\d^{2}}{\d s\,\d\rho }|_{s=\rho=0}\sum_{i=1}^{m}\P_{\rho}((X,Y)\in \Omega_{i}^{(s)}\times\Omega_{i}^{(s)})=0.$$
%\snote{Require equality of inequality above?}
In the case $\rho\in(-1,0)$, the inequality is reversed.
\end{definition}
\begin{conj}[\embolden{Plurality is Stablest, Informal Version}, {\cite{khot07}, \cite[Conjecture 1.9]{isaksson11}}]\label{pisinf}
Suppose we run an election with a large number $n$ of voters and $m\geq3$ candidates.  We make the following assumptions about voter behavior and about the election method.
\begin{itemize}
\item Voters cast their votes randomly, independently, with equal probability of voting for each candidate.
\item Each voter has a small influence on the outcome of the election.  (That is, all influences from Definition \ref{infdef} are small for the voting method.)
\item Each candidate has an equal chance of winning the election.
\end{itemize}
Under these assumptions, the plurality function is the voting method that best preserves the outcome of the election, when votes have been corrupted independently each with probability less than $1/2$.
\end{conj}

\subsection{Our Contribution}

Our main contribution is a conditional proof of the continuous version of the Plurality is Stablest Conjecture \ref{pisinf} for $m=3$ or $m=4$ candidates.

\begin{theorem}[\embolden{Conditional Proof of Standard Simplex Conjecture}]\label{main1}
Let $3\leq m\leq4$.  Let $\rho\in(0,1)$.  Let $\Omega_{1},\ldots,\Omega_{m}\subset\R^{\adimn}$ be a partition of Euclidean space with $\P(X\in \Omega_{i})=1/m$ for all $1\leq i\leq m$ such that
$$\sum_{i=1}^{m}\P_{\rho}((X,Y)\in \Omega_{i}\times\Omega_{i})\geq \sum_{i=1}^{m}\P_{\rho}((X,Y)\in \Theta_{i}\times\Theta_{i})$$
for all partitions $\Theta_{1},\ldots,\Theta_{m}\subset\R^{\adimn}$ with $\P(X\in \Theta_{i})=1/m$ for all $1\leq i\leq m$.  Assume additionally that any $\Omega_{1},\ldots,\Omega_{m}$ are hyperstable, as in Definition \ref{hyppart}.  Then $\Omega_{1},\ldots,\Omega_{m}$ are the cones over a regular simplex centered at the origin.
\end{theorem}

We also prove a bilinear version of Theorem \ref{main1} (see Theorem \ref{main16} below), implying an analogue of Theorem \ref{main1} for negative $\rho$.  The case of negative $\rho$ then implies a (conditional) hardness result for MAX-$m$-CUT for $m=3$ and $m=4$ \cite{isaksson11}.

\begin{theorem}[\embolden{Conditional Proof of MAX-$m$-CUT Hardness}]\label{main3}
Let $3\leq m\leq4$.  Under the assumption of Theorem \ref{main1}, the Frieze-Jerrum semidefinite program for MAX-$m$-CUT \cite{frieze95} achieves the best possible quality of approximation in polynomial time, assuming the Unique Games Conjecture.
\end{theorem}

Since the Proof of Theorem \ref{main1} identifies all local maxima of the noise stability (assuming hyperstability of the optimal sets), we can conclude a conditional version of the related Propeller Conjecture of Khot and Naor for four sets.  (The case of $3$ sets follows from an elementary argument, and the case of $4$ sets was proven using computer-assistance in \cite{heilman14}, though our result below avoids any computer assistance.)

\begin{theorem}[\embolden{Conditional Proof of Propeller Conjecture}]\label{main2}
Let $3\leq m\leq 4$.  Let $\Omega_{1},\ldots,\Omega_{m}$ $\subset\R^{\adimn}$ be a partition of Euclidean space with $\P(X\in \Omega_{i})=1/m$ for all $1\leq i\leq m$ such that
$$\sum_{i=1}^{m}\P((X,Y)\in \Omega_{i}\times\Omega_{i})\geq \sum_{i=1}^{m}\P((X,Y)\in \Theta_{i}\times\Theta_{i})$$
for all partitions $\Theta_{1},\ldots,\Theta_{m}\subset\R^{\adimn}$ with $\P(X\in \Theta_{i})=1/m$ for all $1\leq i\leq m$.  Assume additionally that $\Omega_{1},\ldots,\Omega_{m}$ are hyperstable, as in Definition \ref{hyppart}.  Then, for any partition $A_{1},\ldots,A_{m}$ of $\R^{\adimn}$ into measurable sets,
$$\sum_{i=1}^{m}\vnorm{\int_{A_{i}}x\gamma_{\adimn}(x)\,\d x}_{\ell_{2}^{\adimn}}^{2}\leq \frac{9}{8\pi}.$$
And equality holds uniquely (up to rotations and measure zero changes) $A_{1}',A_{2}',A_{3}'\subset\R^{2}$ are three 120 degree sectors centered at the origin, $A_{i}=A_{i}'\times\R^{n-1}$ for all $1\leq i\leq 3$, and $A_{i}=\emptyset$ for all $4\leq i\leq m$.
\end{theorem}

The restriction that $m\leq 4$ in the above statements arises since we are presently unable to classify dilation-invariant hyperstable partitions when $m\geq5$.  See Remark \ref{m5rk} for more on this difficulty.

Lastly, we can show that hyperstable symmetric cylinders are star-shaped.  This result can be compared with an analogous Gaussian surface area result of \cite{heilman22b}.  The latter (unconditional) result says that cylinders that are Gaussian minimal surfaces must be convex (or have a convex complement).

\begin{theorem}[\embolden{Symmetric Hyperstable Cylinders are Star-Shaped}]\label{symthm}
Suppose $\Omega\subset\R^{\adimn}$ is a symmetric hyperstable set and $\exists$ $\Omega_{0}\subset\R^{\sdimn}$ such that
$$\Omega=\Omega_{0}\times\R.$$
%Assume that
%$$\rho\int_{\Sigma}\frac{\d}{\d\rho}T_{\rho}(1_{\Omega})(x)\langle x,N(x)\rangle\gamma_{\adimn}(x)\,\d x\geq0.$$
Then $\Omega$ or $\Omega^{c}$ is star-shaped with respect to the origin.
\end{theorem}

\subsection{Related Work: Level Sets of Caloric Functions and the ``Matzoh ball soup'' Problem}

Let $\Omega\subset\R^{\adimn}$ and let $P_{t}1_{\Omega}$ denote the heat flow applied to $1_{\Omega}$ at time $t>0$, i.e.
$$P_{t}f(x)\colonequals\int_{\R^{\adimn}}f(x+y\sqrt{2t})\gamma_{\adimn}(y)\,\d y,\qquad\forall\, f\colon\R^{\adimn}\to[0,1],$$
so that
$$\frac{\partial}{\partial t}P_{t}f(x)=\sum_{i=1}^{\adimn}\frac{\partial^{2}}{\partial x_{i}^{2}}P_{t}f(x),\qquad\forall\,t>0,\,\forall\,x\in\R^{\adimn}.$$
Suppose $\Omega$ is a superlevel set of $P_{t}1_{\Omega}$, for all $t>0$ (that is, $\Omega=\{x\in\R^{\adimn}\colon P_{t}1_{\Omega}(x)\geq c\}$ for some $c=c(t)\in\R$).  It was then shown in \cite{mag02} that $\Omega$ must be a ball (assuming a priori that $\Omega$ is compact, and it satisfies regularity assumptions: it satisfies an exterior sphere condition and an interior cone condition.)  More generally, if $\Omega$  is bounded, if $\Omega$ satisfies an exterior sphere condition, and if $D\subset\Omega$ with $\overline{D}\subset\Omega$, if $\partial D$ satisfies an interior cone condition, and if $D$ is a superlevel set of $P_{t}1_{D}$ for all $t>0$, then $\Omega$ must be a ball.  This result of \cite{mag02} solved the so-called ``Matzoh Ball Soup Problem.''

There is an analogy with this problem and our hyperstable assumption for sets optimizing noise stability.  If $\Omega$ maximizes noise stability $\int_{\R^{\adimn}}1_{\Omega}(x)T_{\rho}1_{\Omega}(x)\gamma_{\adimn}(x)\,\d x$ for fixed $\rho>0$, then $\Omega$ is a superlevel set of $T_{\rho}1_{\rho}$, where $T_{\rho}$ is the Ornstein-Uhlenbeck operator from \eqref{oudef}.  (See e.g. Lemma \ref{firstvarmaxns}.)  If $\Omega$ is also a critical point of $\frac{\d}{\d\eta}\Big|_{\eta=\rho}\int_{\R^{\adimn}}1_{\Omega}(x)T_{\eta}1_{\Omega}(x)\gamma_{\adimn}(x)\,\d x$, then $\partial\Omega$ is also a level set of $\frac{\d}{\d\eta}\Big|_{\eta=\rho}T_{\eta}1_{\Omega}(x)$, so that $\partial\Omega$ is a level set of $T_{\eta}1_{\Omega}$ as $\eta$ converges to $\rho$.  This condition is analogous to that of \cite{mag02}, which requires a stronger assumption on the level set for \textit{all} times $t$, rather than just an infinitesimal neighborhood of times.

\subsection{Related Work: Self-Shrinkers for Mean Curvature Flow}

Let $\Sigma\subset\R^{\adimn}$ be a compact $C^{\infty}$ smooth manifold.  Assume that $\Sigma$ is a self-shrinker, so that
\begin{equation}\label{sseq}
H(x)=\langle x,N(x)\rangle,\qquad\forall\,x\in\Sigma.
\end{equation}
Here $N(x)$ is the exterior pointing unit normal vector to $\Sigma$ at $x$ and $H(x)\colonequals\mathrm{div}(N)(x)$ is the mean curvature of $\Sigma$ at $x$.  Self-shrinkers are of interest in the theory of mean curvature flows, since they are stationary under the flow.  In their investigation of mean curvature flow \cite{colding12a}, Colding and Minicozzi studied a maximal version of the Gaussian surface area of an $\sdimn$-dimensional $C^{\infty}$ hypersurface $\Sigma$ in $\R^{\adimn}$.  They called this quantity
\begin{equation}\label{cment}
\sup_{a>0,b\in\R^{\adimn}}\int_{\Sigma}a^{-\frac{\sdimn}{2}}\gamma_{\sdimn}((x-b)a^{-1/2})dx
\end{equation}
the ``entropy'' of $\Sigma$.   It turns out the self-shrinkers are critical points of the entropy functional \cite{colding12a}.  More specifically, taking the second derivative of \eqref{cment} when $a=0$ and $b=1$, one obtains $-\int_{\Sigma}f(x)Lf(x)\gamma_{\adimn}(x)\,\d x$, where the operator $L$ is defined on smooth compactly supported functions $f$ on $\Sigma$:
\begin{equation}\label{elleq}
L\colonequals\Delta-\langle x,\nabla\rangle+\vnorm{A}^{2}+1.
\end{equation}
Here $\vnorm{A}^{2}$ is the sum of the squares of all second order partial derivatives of the unit normal vector $N$ at $x$.  (That is, $A$ is the second fundamental form.)  Also, $\Delta$ is the Laplacian on $\Sigma$ (the sum of repeated second derivatives), and $\nabla$ is the gradient on $\Sigma$.

For self-shrinkers \eqref{sseq}, it was a key insight of \cite{colding12a} that $L$ has two distinct eigenspaces with positive eigenvalues.  The first eigenspace has eigenvalue one.  For any $v\in\R^{\adimn}$, the function $x\mapsto \langle v,N(x)\rangle$ satisfies
\begin{equation}\label{lvneq}
L\langle v,N\rangle=\langle v,N\rangle.
\end{equation}
The second eigenspace has eigenvalue two.  The function $f(x)\colonequals \langle x,N(x)\rangle$ satisfies
\begin{equation}\label{lheq}
Lf=2f.
\end{equation}
We note in passing that the function $x\mapsto \langle v,N(x)\rangle$ corresponds to an infinitesimal translation of the surface $\Sigma$, and the function $x\mapsto \langle x,N(x)\rangle$ corresponds to an infinitesimal dilation of the surface $\Sigma$.  Also, the eigenvalues written in \cite{colding12a} are half of those written here, since our Gaussian measure has an exponent that divides by $2$ whereas theirs has a division by $4$.

If $H$ does not changes sign on $\Sigma$, then $\Sigma$ satisfying \eqref{sseq} must be a round cylinder\cite{huisken90}.  If $H$ changes sign on $\Sigma$, then $\langle x,N(x)\rangle$ changes sign on $\Sigma$.  So \eqref{lheq} and Courant's Nodal Domain theorem imply there is yet another eigenfunction of $L$ with eigenvalue greater than $2$, demonstrating that $\Sigma$ is an unstable critical point for the functional \eqref{cment}.  So, in either case, one can classify the compact smooth self-shrinkers that are stable critical points of the entropy \eqref{cment} \cite{colding12a}.

It was noted in \cite{mcgonagle15} that, if the surface $\Sigma$ satisfies the more general equation
\begin{equation}\label{sseqv2}
H(x)=\langle x,N(x)\rangle+\lambda,\qquad\forall\,x\in\Sigma
\end{equation}
for some $\lambda\in\R$, then \eqref{lvneq} still holds, but unfortunately \eqref{lheq} no longer holds \cite[Lemma 4.1]{heilman17}.  If $f(x)\colonequals \langle x,N(x)\rangle$ and if \eqref{sseqv2} holds, then
\begin{equation}\label{lheqv2}
Lf=2f+\lambda.
\end{equation}
%LH=2H+lam|A|^2
%L<xN>= L (H-lam) = LH-L lam  = 2H+lam|A|^2   - lam|A|^2 - lam  = 2H-lam = 2(<xN>+lam)-lam  = 2<xN>+lam
The realization that \eqref{sseqv2} implies \eqref{lvneq} led \cite{mcgonagle15} to classify half spaces as the only stable critical points of the Gaussian surface area functional.  An analogous realization for partitions of Euclidean space was one crucial ingredient in the proofs of the Gaussian double bubble and multi-bubble conjecture \cite{heilman18,milman18b,heilman18b}.

In our noise stability problem, there are analogues of \eqref{lvneq} and \eqref{lheqv2} for a single set and for partitions of Euclidean space into multiple sets.  For simplicity of exposition, we focus for now on one set.  Since the Gaussian surface area functional can be obtained as an appropriate limit of noise stability \cite{ledoux94}, it is not unreasonable that one could generalize \eqref{lvneq} and \eqref{lheqv2} to the setting of noise stability.  However, it was not obvious how to obtain such a generalization, and it first appeared in \cite{heilman20d,heilman21}.  To generalize \eqref{lvneq} and \eqref{lheqv2}, we replace the Colding-Minicozzi entropy functional \eqref{cment} with the noise stability of a measurable Euclidean set $\Omega\subset\R^{\adimn}$ with correlation $0<\rho<1$:
\begin{equation}\label{noisestabeq}
\int_{\R^{\adimn}}1_{\Omega}(x)T_{\rho}1_{\Omega}(x)\gamma_{\adimn}(x)\,\d x.
\end{equation}
(The Ornstein-Uhleneck operator $T_{\rho}$ is defined in \eqref{oudef}.  Also, it is a standard exercise to show that \eqref{noisestabeq} is equal to $\P_{\rho}((XY)\in\Omega\times\Omega)$, i.e. this definition agrees with our previous definition of noise stability \eqref{zero0}.)  When we take the second derivative of this functional, we obtain the quantity
$$\int_{\Sigma}\int_{\Sigma}(S(f)(x)- f(x)\vnormf{\overline{\nabla }T_{\rho}1_{\Omega}(x)})f(x)\gamma_{\adimn}(x)\,\d x,$$
where $S$ is the second variation operator analogous to \eqref{elleq}, defined for $C^{\infty}$ compactly supported functions $f\colon\Sigma\to\R$.
$$
S(f)(x)\colonequals (1-\rho^{2})^{-(\adimn)/2}(2\pi)^{-(\adimn)/2}\int_{\Sigma}f(y)e^{-\frac{\vnorm{y-\rho x}^{2}}{2(1-\rho^{2})}}\,\d y,\qquad\forall\,x\in\Sigma\colonequals\partial\Omega.
$$
For critical points $\Omega$ of the noise stability functional \eqref{noisestabeq}, if $\Sigma\colonequals\partial\Omega$, then there is an analogue of \eqref{lvneq}:
$$
S(\langle v,N\rangle)(x)= \frac{1}{\rho}\langle v,N(x)\rangle\vnormf{\overline{\nabla }T_{\rho}1_{\Omega}(x)},\qquad\forall\,x\in\Sigma.
$$
In \eqref{lvneq} there was an eigenvalue of $1$, which has now become an eigenvalue of $1/\rho$.

However, the analogy with \eqref{lheq} is imperfect.  In the setting of Gaussian surface area, we know \cite{colding12a} that \eqref{sseq} implies \eqref{lheq}.  Yet, the assumption for noise stability that is analogous to \eqref{sseq} is unclear.  Setting $\lambda=0$ in \eqref{sseqv2} results in an eigenfunction in \eqref{lheqv2}, and this is a key fact in \cite{colding12a}.  Analogously, a critical point of noise stability satisfies $T_{\rho}1_{\Omega}(x)=c$ for all $x\in\partial\Omega$ (by Lemma \ref{firstvarmaxns}), and there is no clear way to impose e.g. a constraint on $c$ that results in a corresponding eigenfunction equation for the second variation operator $S$.  As a remedy for this issue, below we will argue that the appropriate analogue of the condition \eqref{sseq} for noise stability is a hyperstable set.

A critical point of the Gaussian surface area functional satisfies \eqref{sseqv2} and the almost-eigenfunction equation \eqref{lheqv2} holds.  Similarly, a critical point $\Omega\subset\R^{\adimn}$ of the noise stability functional satisfies the almost-eigenfunction equation (see \eqref{six1zv} below)
\begin{equation}\label{sf1eq}
S(f)(x)=\frac{1}{\rho^{2}}f(x)\vnormf{\overline{\nabla }T_{\rho}1_{\Omega}(x)}+\Big(\frac{1}{\rho^{2}}-1\Big)\rho\frac{\d}{\d\rho}T_{\rho}1_{\Omega}(x),\qquad\forall\,x\in\Sigma,
\end{equation}
where $f(x)\colonequals\langle x,N(x)\rangle$.  Without the derivative term on the right, this equation would say that $f$ is an eigenfunction of $S$ with eigenvalue $1/\rho^{2}$.  Similarly, when $\lambda=0$ in \eqref{sseqv2}, \eqref{lheqv2} says that $f$ is an eigenfunction of $L$ with eigenvalue $2$.  So, conditions guaranteeing that the derivative term in \eqref{sf1eq} are zero are analogous to the self-shrinker equation \eqref{sseq}.

In summary, the hyperstability condition for the noise stability functional is a natural analogue of the self-shrinker condition \eqref{sseq} for the Colding-Minicozzi entropy \eqref{cment}.

\subsection{More Formal Introduction}

Using a generalization of the Central Limit Theorem known as the invariance principle \cite{mossel10,isaksson11}, there is an equivalence between the discrete problem of Conjecture \ref{pisinf} and a continuous problem which is known as the Standard Simplex Conjecture \cite{isaksson11}.  For more details on this equivalence, see Section 7 of \cite{isaksson11}We begin by providing some background for the latter conjecture, stated in Conjecture \ref{conj2} below.

For any $k\geq1$, we define the Gaussian density as
\begin{equation}\label{zero0.0}
\begin{aligned}
\gamma_{k}(x)&\colonequals (2\pi)^{-k/2}e^{-\vnormt{x}^{2}/2},\qquad
\langle x,y\rangle\colonequals\sum_{i=1}^{\adimn}x_{i}y_{i},\qquad
\vnormt{x}^{2}\colonequals\langle x,x\rangle,\\
&\qquad\forall\,x=(x_{1},\ldots,x_{\adimn}),y=(y_{1},\ldots,y_{\adimn})\in\R^{\adimn}.
\end{aligned}
\end{equation}

Let $z_{1},\ldots,z_{m}\in\R^{\adimn}$ be the vertices of a regular simplex in $\R^{\adimn}$ centered at the origin.  For any $1\leq i\leq m$, define
\begin{equation}\label{wdef}
\Omega_{i}\colonequals\{x\in\R^{\adimn}\colon\langle x,z_{i}\rangle=\max_{1\leq j\leq m}\langle x,z_{j}\rangle\}.
\end{equation}
We refer to any sets satisfying \eqref{wdef} as \textbf{cones over a regular simplex}.

Let $f\colon\R^{\adimn}\to[0,1]$ be measurable and let $\rho\in(-1,1)$.  Define the \textbf{Ornstein-Uhlenbeck operator with correlation $\rho$} applied to $f$ by
\begin{equation}\label{oudef}
\begin{aligned}
T_{\rho}f(x)
&\colonequals\int_{\R^{\adimn}}f(x\rho+y\sqrt{1-\rho^{2}})\gamma_{\adimn}(y)\,\d y\\
&=(1-\rho^{2})^{-(\adimn)/2}(2\pi)^{-(\adimn)/2}\int_{\R^{\adimn}}f(y)e^{-\frac{\vnorm{y-\rho x}^{2}}{2(1-\rho^{2})}}\,\d y,
\qquad\forall x\in\R^{\adimn}.
\end{aligned}
\end{equation}
$T_{\rho}$ is a parametrization of the Ornstein-Uhlenbeck operator, which gives a fundamental solution of the (Gaussian) heat equation when $\rho\neq0$:
\begin{equation}\label{oup}
\frac{d}{d\rho}T_{\rho}f(x)=\frac{1}{\rho}\Big(-\overline{\Delta} T_{\rho}f(x)+\langle x,\overline{\nabla}T_{\rho}f(x)\rangle\Big),\qquad\forall\,x\in\R^{\adimn}.
\end{equation}
Here $\overline{\Delta}\colonequals\sum_{i=1}^{\adimn}\partial^{2}/\partial x_{i}^{2}$ and $\overline{\nabla}$ is the usual gradient on $\R^{\adimn}$.  $T_{\rho}$ is not a semigroup, but it satisfies $T_{\rho_{1}}T_{\rho_{2}}=T_{\rho_{1}\rho_{2}}$ for all $\rho_{1},\rho_{2}\in(0,1)$.  We have chosen this definition since the usual Ornstein-Uhlenbeck operator is only defined for $\rho\in[0,1]$.

\begin{definition}[\embolden{Noise Stability}]\label{noisedef}
Let $\Omega\subset\R^{\adimn}$ be measurable.  Let $\rho\in(-1,1)$.  We define the \textit{noise stability} of the set $\Omega$ with correlation $\rho$ to be
$$\int_{\R^{\adimn}}1_{\Omega}(x)T_{\rho}1_{\Omega}(x)\gamma_{\adimn}(x)\,\d x
\stackrel{\eqref{oudef}}{=}(2\pi)^{-(\adimn)}(1-\rho^{2})^{-(\adimn)/2}\int_{\Omega}\int_{\Omega}e^{\frac{-\|x\|^{2}-\|y\|^{2}+2\rho\langle x,y\rangle}{2(1-\rho^{2})}}\,\d x\d y.$$
Equivalently, if $X=(X_{1},\ldots,X_{\adimn}),Y=(Y_{1},\ldots,Y_{\adimn})\in\R^{\adimn}$ are $(\adimn)$-dimensional jointly Gaussian distributed random vectors with $\E X_{i}Y_{j}=\rho\cdot1_{(i=j)}$ for all $i,j\in\{1,\ldots,\adimn\}$, then
$$\int_{\R^{\adimn}}1_{\Omega}(x)T_{\rho}1_{\Omega}(x)\gamma_{\adimn}(x)\,\d x=\mathbb{P}((X,Y)\in \Omega\times \Omega).$$
\end{definition}

Maximizing the noise stability of a Euclidean partition is the continuous analogue of finding a voting method that is most stable to random corruption of votes, among voting methods where each voter has a small influence on the election's outcome.

\begin{prob}[\embolden{Standard Simplex Problem}, {\cite{isaksson11}}]\label{prob2}
Let $m\geq3$.  Fix $a_{1},\ldots,a_{m}>0$ such that $\sum_{i=1}^{m}a_{i}=1$.  Fix $\rho\in(0,1)$.  Find measurable sets $\Omega_{1},\ldots\Omega_{m}\subset\R^{\adimn}$ with $\cup_{i=1}^{m}\Omega_{i}=\R^{\adimn}$ and $\gamma_{\adimn}(\Omega_{i})=a_{i}$ for all $1\leq i\leq m$ that maximize
$$\sum_{i=1}^{m}\int_{\R^{\adimn}}1_{\Omega_{i}}(x)T_{\rho}1_{\Omega_{i}}(x)\gamma_{\adimn}(x)\,\d x,$$
subject to the above constraints.
\end{prob}

We can now state the continuous version of Conjecture \ref{pisinf}.

\begin{conj}[\embolden{Standard Simplex Conjecture} {\cite{isaksson11}}]\label{conj2}
Let $\Omega_{1},\ldots\Omega_{m}\subset\R^{\adimn}$ maximize Problem \ref{prob2}.  Assume that $m-1\leq\adimn$.  Fix $\rho\in(0,1)$.  Let $z_{1},\ldots,z_{m}\in\R^{\adimn}$ be the vertices of a regular simplex in $\R^{\adimn}$ centered at the origin.  Then $\exists$ $w\in\R^{\adimn}$ such that, for all $1\leq i\leq m$,
$$\Omega_{i}=w+\{x\in\R^{\adimn}\colon\langle x,z_{i}\rangle=\max_{1\leq j\leq m}\langle x,z_{j}\rangle\}.$$
\end{conj}
It is known that Conjecture \ref{conj2} is false when $(a_{1},\ldots,a_{m})\neq(1/m,\ldots,1/m)$ \cite{heilman14}.  In the remaining case that $a_{i}=1/m$ for all $1\leq i\leq m$, it is assumed that $w=0$ in Conjecture \ref{conj2}.

%For expositional simplicity, we separately address the case $\rho<0$ of Conjecture \ref{conj2} in Section \ref{negsec} below.

\subsection{Plurality is Stablest Conjecture}

As previously mentioned, the Standard Simplex Conjecture \cite{isaksson11} stated in Conjecture \ref{conj2} is essentially equivalent to the Plurality is Stablest Conjecture from Conjecture \ref{pisinf}.   After making several definitions, we state a formal version of Conjecture \ref{pisinf} as Conjecture \ref{prob4} below.

If $g\colon\{1,\ldots,m\}^{\sdimn}\to\R$ and $1\leq i\leq\sdimn$, we denote
$$ \E(g)\colonequals m^{-\sdimn}\sum_{\omega\in\{1,\ldots,m\}^{\sdimn}} g(\omega)$$
$$\E_{i}(g)(\omega_{1},\ldots,\omega_{i-1},\omega_{i+1},\ldots,\omega_{\sdimn})\colonequals m^{-1}\sum_{\omega_{i}\in\{1,\ldots,m\}} g(\omega_{1},\ldots,\omega_{n})$$
$$\qquad\qquad\qquad\qquad\qquad\qquad\qquad\qquad\qquad\forall\,(\omega_{1},\ldots,\omega_{i-1},\omega_{i+1},\ldots,\omega_{\sdimn})\in\{1,\ldots,m\}^{\sdimn}.$$
Define also the $i^{th}$ \textbf{influence} of $g$, i.e. the influence of the $i^{th}$ voter of $g$, as
\begin{equation}\label{infdef}
\mathrm{Inf}_{i}(g)\colonequals \E [(g-\E_{i}g)^{2}].
\end{equation}
Let
\begin{equation}\label{deltadef}
\Delta_{m}\colonequals\{(y_{1},\ldots,y_{m})\in\R^{m}\colon y_{1}+\cdots+y_{m}=1,\,\forall\,1\leq i\leq m,\,y_{i}\geq0\}.
\end{equation}
If $f\colon\{1,\ldots,m\}^{\sdimn}\to\Delta_{m}$, we denote the coordinates of $f$ as $f=(f_{1},\ldots,f_{m})$.  For any $\omega\in\Z^{\sdimn}$, we denote $\vnormt{\omega}_{0}$ as the number of nonzero coordinates of $\omega$.  The \textbf{noise stability} of $g\colon\{1,\ldots,m\}^{\sdimn}\to\R$ with parameter $\rho\in(-1,1)$ is
\begin{flalign*}
S_{\rho} g
&\colonequals m^{-\sdimn}\sum_{\omega\in\{1,\ldots,m\}^{\sdimn}} g(\omega)\E_{\rho} g(\delta)\\
&=m^{-\sdimn}\sum_{\omega\in\{1,\ldots,m\}^{\sdimn}} g(\omega)\sum_{\sigma\in\{1,\ldots,m\}^{\sdimn}}\left(\frac{1-(m-1)\rho}{m}\right)^{\sdimn-\vnormt{\sigma-\omega}_{0}}
\left(\frac{1-\rho}{m}\right)^{\vnormt{\sigma-\omega}_{0}} g(\sigma).
\end{flalign*}
Equivalently, conditional on $\omega$, $\E_{\rho}g(\delta)$ is defined so that for all $1\leq i\leq\sdimn$, $\delta_{i}=\omega_{i}$ with probability $\frac{1-(m-1)\rho}{m}$, and $\delta_{i}$ is equal to any of the other $(m-1)$ elements of $\{1,\ldots,m\}$ each with probability $\frac{1-\rho}{m}$, and so that $\delta_{1},\ldots,\delta_{\sdimn}$ are independent.

The \textbf{noise stability} of $f\colon\{1,\ldots,m\}^{\sdimn}\to\Delta_{m}$ with parameter $\rho\in(-1,1)$ is
$$S_{\rho}f\colonequals\sum_{i=1}^{m}S_{\rho}f_{i}.$$

Let $m\geq2$, $k\geq3$.  For each $j\in\{1,\ldots,m\}$, let $e_{j}=(0,\ldots,0,1,0,\ldots,0)\in\R^{m}$ be the $j^{th}$ unit coordinate vector.  Define the \textbf{plurality} function $\mathrm{PLUR}_{m,\sdimn}\colon\{1,\ldots,m\}^{\sdimn}\to\Delta_{m}$ for $m$ candidates and $\sdimn$ voters such that for all $\omega\in\{1,\ldots,m\}^{\sdimn}$.
$$\mathrm{PLUR}_{m,\sdimn}(\omega)
\colonequals\begin{cases}
e_{j}&,\mbox{if }\abs{\{i\in\{1,\ldots,m\}\colon\omega_{i}=j\}}>\abs{\{i\in\{1,\ldots,m\}\colon\omega_{i}=r\}},\\
&\qquad\qquad\qquad\qquad\forall\,r\in\{1,\ldots,m\}\setminus\{j\}\\
\frac{1}{m}\sum_{i=1}^{m}e_{i}&,\mbox{otherwise}.
\end{cases}
$$

We can now state the more formal version of Conjecture \ref{pisinf}.

\begin{conj}[\embolden{Plurality is Stablest, Discrete Version}]\label{prob4}
For any $m\geq2$, $\rho\in[0,1]$, $\epsilon>0$, there exists $\tau>0$ such that if $f\colon\{1,\ldots,m\}^{\sdimn}\to\Delta_{m}$ satisfies $\mathrm{Inf}_{i}(f_{j})\leq\tau$ for all $1\leq i\leq\sdimn$ and for all $1\leq j\leq m$, and if $\E f=\frac{1}{m}\sum_{i=1}^{m}e_{i}$, then
$$
S_{\rho}f\leq \lim_{\sdimn\to\infty}S_{\rho}\mathrm{PLUR}_{m,\sdimn}+\epsilon.
$$
\end{conj}

One main result of \cite{heilman20d} is: $\exists$ $\rho_{0}>0$ such that Conjecture \ref{prob4} is true for $m=3$ for all $0<\rho<\rho_{0}$, for all $n\geq1$.  The only previously known case of Conjecture \ref{prob4} was the following.

\begin{theorem}[\embolden{Majority is Stablest, Formal, Biased Case}, {\cite[Theorem 4.4]{mossel10}}]\label{misg}
Conjecture \ref{prob4} is true when $m=2$.
\end{theorem}

For an even more general version of Theorem \ref{misg}, see \cite[Theorem 4.4]{mossel10}.  In particular, the assumption on $\E f$ can be removed, though we know this cannot be done for $m\geq3$ \cite{heilman14}.

\subsection{Outline of the Proof of the Structure Theorem}

\section{Existence and Regularity}

\subsection{Preliminaries and Notation}\label{secpre}

We say that $\Sigma\subset\R^{\adimn}$ is an $\sdimn$-dimensional $C^{\infty}$ manifold with boundary if $\Sigma$ can be locally written as the graph of a $C^{\infty}$ function on a relatively open subset of $\{(x_{1},\ldots,x_{\sdimn})\in\R^{\sdimn}\colon x_{\sdimn}\geq0\}$.  For any $(\adimn)$-dimensional $C^{\infty}$ manifold $\Omega\subset\R^{\adimn}$ such that $\partial\Omega$ itself has a boundary, we denote
\begin{equation}\label{c0def}
\begin{aligned}
C_{0}^{\infty}(\Omega;\R^{\adimn})
&\colonequals\{f\colon \Omega\to\R^{\adimn}\colon f\in C^{\infty}(\Omega;\R^{\adimn}),\, f(\partial\partial \Omega)=0,\\
&\qquad\qquad\qquad\exists\,r>0,\,f(\Omega\cap(B(0,r))^{c})=0\}.
\end{aligned}
\end{equation}
We also denote $C_{0}^{\infty}(\Omega)\colonequals C_{0}^{\infty}(\Omega;\R)$.  We let $\mathrm{div}$ denote the divergence of a vector field in $\R^{\adimn}$.  For any $r>0$ and for any $x\in\R^{\adimn}$, we let $B(x,r)\colonequals\{y\in\R^{\adimn}\colon\vnormt{x-y}\leq r\}$ be the closed Euclidean ball of radius $r$ centered at $x\in\R^{\adimn}$.  Here $\partial\partial\Omega$ refers to the $(\sdimn-1)$-dimensional boundary of $\Omega$.

\begin{definition}[\embolden{Reduced Boundary}]\label{rbdef}
A measurable set $\Omega\subset\R^{\adimn}$ has \embolden{locally finite surface area} if, for any $r>0$,
$$\sup\left\{\int_{\Omega}\mathrm{div}(X(x))\,\d x\colon X\in C_{0}^{\infty}(B(0,r),\R^{\adimn}),\, \sup_{x\in\R^{\adimn}}\vnormt{X(x)}\leq1\right\}<\infty.$$
Equivalently, $\Omega$ has locally finite surface area if $\nabla 1_{\Omega}$ is a vector-valued Radon measure such that, for any $x\in\R^{\adimn}$, the total variation
$$
\vnormt{\nabla 1_{\Omega}}(B(x,1))
\colonequals\sup_{\substack{\mathrm{partitions}\\ C_{1},\ldots,C_{m}\,\mathrm{of}\,B(x,1) \\ m\geq1}}\sum_{i=1}^{m}\vnormt{\nabla 1_{\Omega}(C_{i})}
$$
is finite \cite{cicalese12}.  If $\Omega\subset\R^{\adimn}$ has locally finite surface area, we define the \embolden{reduced boundary} $\redb \Omega$ of $\Omega$ to be the set of points $x\in\R^{\adimn}$ such that
$$N(x)\colonequals-\lim_{r\to0^{+}}\frac{\nabla 1_{\Omega}(B(x,r))}{\vnormt{\nabla 1_{\Omega}}(B(x,r))}$$
exists, and it is exactly one element of $S^{\sdimn}\colonequals\{x\in\R^{\adimn}\colon\vnorm{x}=1\}$.
\end{definition}

The reduced boundary $\redb\Omega$ is a subset of the topological boundary $\partial\Omega$.  Also, $\redb\Omega$ and $\partial\Omega$ coincide with the support of $\nabla 1_{\Omega}$, except for a set of $\sdimn$-dimensional Hausdorff measure zero.

Let $\Omega\subset\R^{\adimn}$ be an $(\adimn)$-dimensional $C^{2}$ submanifold with reduced boundary $\Sigma\colonequals\redb \Omega$.  Let $N\colon\redA\to S^{\sdimn}$ be the unit exterior normal to $\redA$.  Let $X\in C_{0}^{\infty}(\R^{\adimn},\R^{\adimn})$.  We write $X$ in its components as $X=(X_{1},\ldots,X_{\adimn})$, so that $\mathrm{div}X=\sum_{i=1}^{\adimn}\frac{\partial}{\partial x_{i}}X_{i}$.  Let $\Psi\colon\R^{\adimn}\times(-1,1)\to\R^{\adimn}$ such that
\begin{equation}\label{nine2.3}
\Psi(x,0)=x,\qquad\qquad\frac{\d}{\d s}\Psi(x,s)=X(\Psi(x,s)),\quad\forall\,x\in\R^{\adimn},\,s\in(-1,1).
\end{equation}
For any $s\in(-1,1)$, let $\Omega^{(s)}\colonequals\Psi(\Omega,s)$.  Note that $\Omega^{(0)}=\Omega$.  Let $\Sigma^{(s)}\colonequals\redb\Omega^{(s)}$, $\forall$ $s\in(-1,1)$.
\begin{definition}
We call $\{\Omega^{(s)}\}_{s\in(-1,1)}$ as defined above a \embolden{variation} of $\Omega\subset\R^{\adimn}$.  We also call $\{\Sigma^{(s)}\}_{s\in(-1,1)}$ a \embolden{variation} of $\Sigma=\redb\Omega$.
\end{definition}

For any $x\in\R^{\adimn}$ and any $s\in(-1,1)$, define
\begin{equation}\label{two9c}
V(x,s)\colonequals\int_{\Omega^{(s)}}G(x,y)\,\d y.
\end{equation}

Below, when appropriate, we let $\,\d x$ denote Lebesgue measure, restricted to a surface $\redA\subset\R^{\adimn}$.

\begin{lemma}[\embolden{Existence of a Maximizer}]\label{existlem}
Let $0<\rho<1$ and let $m\geq2$.  Then there exist measurable sets $\Omega_{1},\ldots,\Omega_{m}$ maximizing Problem \ref{prob2}.
\end{lemma}%

\begin{lemma}[\embolden{Regularity of a Maximizer}]\label{reglem}
Let $\Omega_{1},\ldots,\Omega_{m}\subset\R^{\adimn}$ be the measurable sets maximizing Problem \ref{prob2}, guaranteed to exist by Lemma \ref{existlem}.  Then the sets $\Omega_{1},\ldots,\Omega_{m}$ have locally finite surface area.  Moreover, for all $1\leq i\leq m$ and for all $x\in\partial\Omega_{i}$, there exists a neighborhood $U$ of $x$ such that $U\cap \partial\Omega_{i}$ is a finite union of $C^{\infty}$ $\sdimn$-dimensional manifolds.
\end{lemma}%

From Lemma \ref{reglem} and Definition \ref{rbdef}, for all $1\leq i<j\leq m$, if $x\in\Sigma_{ij}$, then the unit normal vector $N_{ij}(x)\in\R^{\adimn}$ that points from $\Omega_{i}$ into $\Omega_{j}$ is well-defined on $\Sigma_{ij}$, $\big((\partial\Omega_{i})\cap(\partial\Omega_{j})\big)\setminus\Sigma_{ij}$ has Hausdorff dimension at most $\sdimn-1$, and
\begin{equation}\label{zero11}
N_{ij}(x)=\pm\frac{\overline{\nabla} T_{\rho}(1_{\Omega_{i}}-1_{\Omega_{j}})(x)}{\vnorm{\overline{\nabla} T_{\rho}(1_{\Omega_{i}}-1_{\Omega_{j}})(x)}},\qquad\forall\,x\in\Sigma_{ij}.
\end{equation}
In Lemma \ref{lemma7p} below we will show that the negative sign holds in \eqref{zero11} when $\Omega_{1},\ldots,\Omega_{m}$ maximize Problem \ref{prob2}.

%also follows from discussion on page 608 here  i think
% file:///C:/Users/Steve/Downloads/Chen1998_Article_AStrongUniqueContinuationTheor%20(1).pdf

%Suppose $u\colon\R^{\adimn}\times[0,\infty)\to\R$ satisfies
%$$\frac{\partial u}{\partial t}(x,t)=\Delta u(x)-\langle x,\overline{\nabla} u\rangle ,\qquad u(x,0)=1_{\Omega}(x),\qquad\forall\,(x,t)\in\R^{\adimn}\times[0,\infty).$$
%\snote{first expand $u$ in hermite polynomials?}
%Set $w(x,t)\colonequals u(xe^{t},t)$ for any $(x,t)\in\R^{\adimn}\times[0,\infty)$.  A short calculation shows that $w$ satisfies
%$$e^{2t}\frac{\partial w}{\partial t}(x,t)=\Delta w,\qquad w(x,0)=f(x),\qquad\forall\,(x,t)\in\R^{\adimn}\times[0,\infty)$$
%So, letting $h(x,\tau)\colonequals w(x,t)$ with $\tau=(1-e^{-2t})/2$ shows that
%$$\frac{\partial h}{\partial t}=\Delta h,\qquad h(x,0)=f(x),\qquad\forall\,(x,\tau )\in\R^{\adimn}\times[0,\infty)$$
%
%\snote{now transform it into an elliptic equation?}
%Thus, using the fundamental solution of the heat equation, we see that
%$h(x,\tau)=\frac{1}{(4\pi \tau)^{n/2}}\int_{\R^{n}}e^{-\frac{\abs{x-y}^{2}}{4\tau}}f(y)\,\d y$, so $w(x,t)=\frac{1}{(4\pi (1-e^{-2t})/2)^{n/2}}\int_{\R^{n}}e^{-\frac{\abs{x-y}^{2}}{4(1-e^{-2t})/2}}f(y)\,\d y$,
%$$u(x,t)=w(xe^{-t},t)=\frac{1}{(2\pi(1-e^{-2t}))^{n/2}}\int_{\R^{n}}e^{-\frac{\absf{xe^{-t}-y}^{2}}{2(1-e^{-2t})}}f(y)\,\d y$$
%or equivalently (by shifting and dilating), we get
%$$u(x,t)=P_{t}f(x)=\int_{\R^{n}}f(e^{-t}x+(1-e^{-2t})^{1/2}y)d\gamma_{n}(y),\quad x\in\R^{n},t\geq0$$

\section{First and Second Variation}

In this section, we recall some standard facts for variations of sets with respect to the Gaussian measure.  Here is a summary of notation.

\textbf{Summary of Notation}.
\begin{itemize}
\item $T_{\rho}$ denotes the Ornstein-Uhlenbeck operator with correlation parameter $\rho\in(-1,1)$.
\item $\Omega_{1},\ldots,\Omega_{m}$ denotes a partition of $\R^{\adimn}$ into $m$ disjoint measurable sets.
\item $\redb\Omega$ denotes the reduced boundary of $\Omega\subset\R^{\adimn}$.
\item $\Sigma_{ij}\colonequals(\redb\Omega_{i})\cap(\redb\Omega_{j})$ for all $1\leq i,j\leq m$.
\item $N_{ij}(x)$ is the unit normal vector to $x\in\Sigma_{ij}$ that points from $\Omega_{i}$ into $\Omega_{j}$, so that $N_{ij}=-N_{ji}$.
\end{itemize}
Throughout the paper, unless otherwise stated, we define $G\colon\R^{\adimn}\times\R^{\adimn}\to\R$ to be the following function.  For all $x,y\in\R^{\adimn}$, $\forall$ $\rho\in(-1,1)$, define
\begin{equation}\label{gdef}
\begin{aligned}
G(x,y)&=(1-\rho^{2})^{-(\adimn)/2}(2\pi)^{-(\adimn)}e^{\frac{-\|x\|^{2}-\|y\|^{2}+2\rho\langle x,y\rangle}{2(1-\rho^{2})}}\\
&=(1-\rho^{2})^{-(\adimn)/2}\gamma_{\adimn}(x)\gamma_{\adimn}(y)e^{\frac{-\rho^{2}(\|x\|^{2}+\|y\|^{2})+2\rho\langle x,y\rangle}{2(1-\rho^{2})}}\\
&=(1-\rho^{2})^{-(\adimn)/2}(2\pi)^{-(\adimn)/2}\gamma_{\adimn}(x)e^{\frac{-\vnorm{y-\rho x}^{2}}{2(1-\rho^{2})}}.
\end{aligned}
\end{equation}

We can then rewrite the noise stability from Definition \ref{noisedef} as
$$\int_{\R^{\adimn}}1_{\Omega}(x)T_{\rho}1_{\Omega}(x)\gamma_{\adimn}(x)\,\d x
=\int_{\Omega}\int_{\Omega}G(x,y)\,\d x\d y.$$
Our first and second variation formulas for the noise stability will be written in terms of $G$.

%also a related result: https://sites.math.washington.edu/~blwilson/Nodal/FHLinNodalSets.pdf

\begin{lemma}[\embolden{The First Variation}\,{\cite{chokski07}}; also {\cite[Lemma 3.1, Equation (7)]{heilman14}}]\label{latelemma3}
Let $X\in C_{0}^{\infty}(\R^{\adimn},\R^{\adimn})$.  Let $\Omega\subset\R^{\adimn}$ be a measurable set such that $\partial\Omega$ is a locally finite union of $C^{\infty}$ manifolds.  Let $\{\Omega^{(s)}\}_{s\in(-1,1)}$ be the corresponding variation of $\Omega$.  Then
\begin{equation}\label{Bone6}
\frac{\d}{\d s}\Big|_{s=0}\int_{\R^{\adimn}} 1_{\Omega^{(s)}}(y)G(x,y)\,\d y
=\int_{\partial \Omega}G(x,y)\langle X(y),N(y)\rangle \,\d y.
\end{equation}
\end{lemma}

The following Lemma is a consequence of \eqref{Bone6} and Lemma \ref{reglem}.

\begin{lemma}[\embolden{The First Variation for Maximizers}]\label{firstvarmaxns}
Suppose $\Omega_{1},\ldots,\Omega_{m}\subset\R^{\adimn}$ maximize Problem \ref{prob2}.  Then for all $1\leq i<j\leq m$, there exists $c_{ij}\in\R$ such that
$$T_{\rho}(1_{\Omega_{i}}-1_{\Omega_{j}})(x)=c_{ij},\qquad\forall\,x\in\Sigma_{ij}.$$
\end{lemma}%
%\begin{proof}
%Fix $1\leq i<j\leq m$ and denote $f_{ij}(x)\colonequals\langle X(x),N_{ij}(x)\rangle$ for all $x\in\Sigma_{ij}$.  From Lemma \ref{latelemma3}, if $X$ is nonzero outside of $\Sigma_{ij}$, we get
%\begin{flalign*}
%&\frac{1}{2}\frac{\d}{\d s}\Big|_{s=0}\sum_{i=1}^{m}\int_{\R^{\adimn}}1_{\Omega_{i}^{(s)}}(x)T_{\rho}1_{\Omega_{i}^{(s)}}(x)\gamma_{\adimn}(x)\,\d x\\
%&\qquad=\int_{\Omega_{i}}G(x,y)\int_{\Sigma_{ij}}\langle X(x),N_{ij}(x)\rangle \,\d x \,\d y
%+\int_{\Omega_{j}}G(x,y)\int_{\Sigma_{ij}}\langle X(x),N_{ji}(x)\rangle \,\d x \,\d y\\
%&\qquad\stackrel{\eqref{oudef}\wedge\eqref{gdef}}{=}\int_{\Sigma_{ij}}T_{\rho}(1_{\Omega_{i}}-1_{\Omega_{j}})(x)f_{ij}(x)\,\d x.
%\end{flalign*}
%We used above $N_{ij}=-N_{ji}$.  If $T_{\rho}(1_{\Omega_{i}}-1_{\Omega_{j}})(x)$ is nonconstant, then we can construct $f_{ij}$ supported in $\Sigma_{ij}$ with $\int_{\redb\Omega_{i'}}f_{ij}(x)\gamma_{\adimn}(x)dx=0$ for all $1\leq i'\leq m$ to give a nonzero derivative, contradicting the maximality of $\Omega_{1},\ldots,\Omega_{m}$ (as in Lemma \ref{reglem} and \eqref{zero9.0}).
%\end{proof}

\begin{theorem}[\embolden{General Second Variation Formula}, {\cite[Theorem 2.6]{chokski07}}; also {\cite[Theorem 1.10]{heilman15}}]\label{thm4}
Let $X\in C_{0}^{\infty}(\R^{\adimn},\R^{\adimn})$.  Let $\Omega\subset\R^{\adimn}$  be a measurable set such that $\partial\Omega$ is a locally finite union of $C^{\infty}$ manifolds.  Let $\{\Omega^{(s)}\}_{s\in(-1,1)}$ be the corresponding variation of $\Omega$.  Define $V$ as in \eqref{two9c}.  Then
\begin{flalign*}
&\frac{1}{2}\frac{\d^{2}}{\d s^{2}}\Big|_{s=0}\int_{\R^{\adimn}} \int_{\R^{\adimn}} 1_{\Omega^{(s)}}(y)G(x,y) 1_{\Omega^{(s)}}(x)\,\d x\d y\\
&\quad=\int_{\redA}\int_{\redA}G(x,y)\langle X(x),N(x)\rangle\langle X(y),N(y)\rangle \,\d x\d y
+\int_{\redA}\mathrm{div}(V(x,0)X(x))\langle X(x),N(x)\rangle \,\d x.
\end{flalign*}

\end{theorem}

\section{Noise Stability and the Calculus of Variations}\label{secnoise}

We now further refine the first and second variation formulas from the previous section.  The following formula follows by using $G(x,y)\colonequals\gamma_{\adimn}(x)\gamma_{\adimn}(y)$ $\forall$ $x,y\in\R^{\adimn}$ in Lemma \ref{latelemma3} and in Theorem \ref{thm4}.

\begin{lemma}[\embolden{Variations of Gaussian Volume}, {\cite{ledoux01}}]\label{lemma41}
Let $\Omega\subset\R^{\adimn}$  be a measurable set such that $\partial\Omega$ is a locally finite union of $C^{\infty}$ manifolds.  Let $X\in C_{0}^{\infty}(\R^{\adimn},\R^{\adimn})$.  Let $\{\Omega^{(s)}\}_{s\in(-1,1)}$ be the corresponding variation of $\Omega$.  Denote $f(x)\colonequals\langle X(x),N(x)\rangle$ for all $x\in\Sigma\colonequals \redb\Omega $.  Then
$$\frac{\d}{\d s}\Big|_{s=0}\gamma_{\adimn}(\Omega^{(s)})=\int_{\Sigma}f(x)\gamma_{\adimn}(x)\,\d x.$$
$$\frac{\d^{2}}{\d s^{2}}\Big|_{s=0}\gamma_{\adimn}(\Omega^{(s)})=\int_{\Sigma}(\mathrm{div}(X)-\langle X,x\rangle)f(x)\gamma_{\adimn}(x)\,\d x.$$
\end{lemma}

\begin{lemma}[\embolden{Extension Lemma for Existence of Volume-Preserving Variations}, {\cite[Lemma 3.9]{heilman18}}]\label{lemma27}
Let $X'\in C_{0}^{\infty}(\R^{\adimn},\R^{\adimn})$ be a vector field.  Define $f_{ij}\colonequals\langle X',N_{ij}\rangle\in C_{0}^{\infty}(\Sigma_{ij})$ for all $1\leq i<j\leq m$.  If
\begin{equation}\label{eight2}
\forall\,1\leq i\leq m,\quad \sum_{j\in\{1,\ldots,m\}\setminus\{i\}}\int_{\Sigma_{ij}}f_{ij}(x)\gamma_{\sdimn}(x)\,\d x=0,
\end{equation}
then $X'|_{\cup_{1\leq i<j\leq m}\Sigma_{ij}}$ can be extended to a vector field $X\in C_{0}^{\infty}(\R^{\adimn},\R^{\adimn})$ such that the corresponding variations $\{\Omega_{i}^{(s)}\}_{1\leq i\leq m,s\in(-1,1)}$ satisfy
$$\forall\,1\leq i\leq m,\quad\forall\,s\in(-1,1),\quad \gamma_{\adimn}(\Omega_{i}^{(s)})=\gamma_{\adimn}(\Omega_{i}).$$
\end{lemma}

\begin{lemma}\label{gpsd}
Define $G$ as in \eqref{gdef}.  Let $f\colon\Sigma\to\R$ be continous and compactly supported.  Then
$$
\int_{\redA}\int_{\redA}G(x,y)f(x)f(y) \,\d x\d y\geq0.
$$
\end{lemma}%
%\begin{proof}
%If $g\colon\R^{\adimn}\to\R$ is continuous and compactly supported, then it is well known that
%$$
%\int_{\redA}\int_{\redA}G(x,y)g(x)g(y) \,\d x\d y\geq0,
%$$
%since e.g. $\frac{G(x,y)}{\gamma_{\adimn(x)}\gamma_{\adimn}(y)}$ is the Mehler kernel, which can be written as an (infinite-dimensional) positive semidefinite matrix.  That is, there exists an orthonormal basis $\{\psi_{i}\}_{i=1}^{\infty}$ of $L_{2}(\gamma_{\adimn})$ (of Hermite polynomials) and there exists a sequence of nonnegative real numbers $\{\lambda_{i}\}_{i=1}^{\infty}$ such that the following series converges absolutely pointwise:
%$$\frac{G(x,y)}{\gamma_{\adimn}(x)\gamma_{\adimn}(y)}=\sum_{i=1}^{\infty}\lambda_{i}\psi_{i}(x)\psi_{i}(y),\qquad\forall\,x,y\in\R^{\adimn}.$$
%From Mercer's Theorem, this is equivalent to : $\forall$ $p\geq1$, for all $z^{(1)},\ldots,z^{(p)}\in\R^{n}$, for all $\beta_{1},\ldots,\beta_{p}\in\R$,
%$$\sum_{i,j=1}^{p}\beta_{i}\beta_{j}G(z^{(i)},z^{(j)})\geq0.$$
%In particular, this holds for all $z^{(1)},\ldots,z^{(p)}\in\partial\Omega\subset\R^{\adimn}$.  So, the positive semidefinite property carries over (by restriction) to $\partial\Omega$.
%\end{proof}

%
%
\subsection{Two Sets}\label{twossub}
\begin{lemma}[\embolden{Volume Preserving Second Variation of Maximizers}]\label{lemma7p}
Suppose $\Omega,\Omega^{c}\subset\R^{\adimn}$ maximize Problem \ref{prob2} for $0<\rho<1$ and $m=2$.  Let $\{\Omega^{(s)}\}_{s\in(-1,1)}$ be the corresponding variation of $\Omega$.  Denote $f(x)\colonequals\langle X(x),N(x)\rangle$ for all $x\in\Sigma\colonequals \redb\Omega$.  If
$$\int_{\Sigma}f(x)\gamma_{\adimn}(x)\,\d x=0,$$
Then there exists an extension of the vector field $X|_{\Sigma}$ such that the corresponding variation of $\{\Omega^{(s)}\}_{s\in(-1,1)}$ satisfies
\begin{equation}\label{four32p}
\begin{aligned}
&\frac{1}{2}\frac{\d^{2}}{\d s^{2}}\Big|_{s=0}\int_{\R^{\adimn}}\int_{\R^{\adimn}} 1_{\Omega^{(s)}}(y)G(x,y) 1_{\Omega^{(s)}}(x)\,\d x\d y\\
&\qquad\qquad=\int_{\redA}\int_{\redA}G(x,y)f(x)f(y) \,\d x\d y
-\int_{\redA}\vnorm{\overline{\nabla} T_{\rho}1_{\Omega}(x)}(f(x))^{2} \gamma_{\adimn}(x)\,\d x.
\end{aligned}
\end{equation}
Moreover,
\begin{equation}\label{nabeq2}
\overline{\nabla}T_{\rho}1_{\Omega}(x)=-N(x)\vnorm{\overline{\nabla}T_{\rho}1_{\Omega}(x)},\qquad\forall\,x\in\Sigma.
\end{equation}
\end{lemma}

\subsection{More than Two Sets}

We can now generalize Section \ref{twossub} to the case of $m>2$ sets.

\begin{lemma}[\embolden{Second Variation of Noise Stability, Multiple Sets}]\label{lemma6v2}
Let $\Omega_{1},\ldots,\Omega_{m}\subset\R^{\adimn}$ be a partition of $\R^{\adimn}$ into measurable sets such that $\partial\Omega_{i}$ is a locally finite union of $C^{\infty}$ manifolds for all $1\leq i\leq m$.  Let $X\in C_{0}^{\infty}(\R^{\adimn},\R^{\adimn})$.  Let $\{\Omega_{i}^{(s)}\}_{s\in(-1,1)}$ be the corresponding variation of $\Omega_{i}$ for all $1\leq i\leq m$.  Denote $f_{ij}(x)\colonequals\langle X(x),N_{ij}(x)\rangle$ for all $x\in\Sigma_{ij}\colonequals (\redb\Omega_{i})\cap(\redb\Omega_{j})$.  We let $N$ denote the exterior pointing unit normal vector to $\redb\Omega_{i}$ for any $1\leq i\leq m$.  Then
\begin{equation}\label{four30v2}
\begin{aligned}
&\frac{1}{2}\frac{\d^{2}}{\d s^{2}}\Big|_{s=0}\sum_{i=1}^{m}\int_{\R^{\adimn}} \int_{\R^{\adimn}} 1_{\Omega_{i}^{(s)}}(y)G(x,y) 1_{\Omega_{i}^{(s)}}(x)\,\d x\d y\\
&\qquad=\sum_{1\leq i<j\leq m}\int_{\Sigma_{ij}}\Big[\Big(\int_{\redb\Omega_{i}}-\int_{\redb\Omega_{j}}\Big)G(x,y)\langle X(y),N(y)\rangle \,\d y\Big] f_{ij}(x) \,\d x\\
&\qquad\qquad+\int_{\Sigma_{ij}}\langle\overline{\nabla} T_{\rho}(1_{\Omega_{i}}-1_{\Omega_{j}})(x),X(x)\rangle f_{ij}(x) \gamma_{\adimn}(x)\,\d x\\
&\qquad\qquad+\int_{\Sigma_{ij}} T_{\rho}(1_{\Omega_{i}}-1_{\Omega_{j}})(x)\Big(\mathrm{div}(X(x))-\langle X(x),x\rangle\Big)f_{ij}(x)\gamma_{\adimn}(x)\,\d x.
\end{aligned}
\end{equation}
\end{lemma}%%

\begin{lemma}[\embolden{Volume Preserving Second Variation of Maximizers, Multiple Sets}]\label{lemma7r}
Let $\Omega_{1},\ldots,\Omega_{m}\subset\R^{\adimn}$ be a partition of $\R^{\adimn}$ into measurable sets such that $\partial\Omega_{i}$ is a locally finite union of $C^{\infty}$ manifolds for all $1\leq i\leq m$.  Let $X\in C_{0}^{\infty}(\R^{\adimn},\R^{\adimn})$.  Let $\{\Omega_{i}^{(s)}\}_{s\in(-1,1)}$ be the corresponding variation of $\Omega_{i}$ for all $1\leq i\leq m$.  Denote $f_{ij}(x)\colonequals\langle X(x),N_{ij}(x)\rangle$ for all $x\in\Sigma_{ij}\colonequals (\redb\Omega_{i})\cap(\redb\Omega_{j}) $.  We let $N$ denote the exterior pointing unit normal vector to $\redb\Omega_{i}$ for any $1\leq i\leq m$.  Then
\begin{equation}\label{four32pv2}
\begin{aligned}
&\frac{1}{2}\frac{\d^{2}}{\d s^{2}}\Big|_{s=0}\sum_{i=1}^{m}\int_{\R^{\adimn}} \int_{\R^{\adimn}} 1_{\Omega_{i}^{(s)}}(y)G(x,y) 1_{\Omega_{i}^{(s)}}(x)\,\d x\d y\\
&\qquad\qquad\qquad=\sum_{1\leq i<j\leq m}\int_{\Sigma_{ij}}\Big[\Big(\int_{\redb\Omega_{i}}-\int_{\redb\Omega_{j}}\Big)G(x,y)\langle X(y),N(y)\rangle \,\d y\Big] f_{ij}(x) \,\d x\\
&\qquad\qquad\qquad\qquad-\int_{\Sigma_{ij}}\vnorm{\overline{\nabla} T_{\rho}(1_{\Omega_{i}}-1_{\Omega_{j}})(x)}(f_{ij}(x))^{2} \gamma_{\adimn}(x)\,\d x.
\end{aligned}
\end{equation}
Also,
\begin{equation}\label{nabeq3}
\overline{\nabla}T_{\rho}(1_{\Omega_{i}}-1_{\Omega_{j}})(x)=-N_{ij}(x)\vnorm{\overline{\nabla}T_{\rho}(1_{\Omega_{i}}-1_{\Omega_{j}})(x)},\qquad\forall\,x\in\Sigma_{ij}.
\end{equation}
Moreover, $\vnorm{\overline{\nabla} T_{\rho}(1_{\Omega_{i}}-1_{\Omega_{j}})(x)}>0$ for all $x\in\Sigma_{ij}$, except on a set of Hausdorff dimension at most $\sdimn-1$.
\end{lemma}%

\section{Almost Eigenfunctions of the Second Variation}

%
%For didactic purposes, we first consider the case $m=2$, and we then later consider the case $m>2$.
%
\subsection{Two Sets}
Let $\Sigma\colonequals\redb\Omega$.  For any bounded measurable $f\colon\Sigma\to\R$, define the following function (if it exists):
\begin{equation}\label{sdef}
S(f)(x)\colonequals (1-\rho^{2})^{-(\adimn)/2}(2\pi)^{-(\adimn)/2}\int_{\Sigma}f(y)e^{-\frac{\vnorm{y-\rho x}^{2}}{2(1-\rho^{2})}}\,\d y,\qquad\forall\,x\in\Sigma.
\end{equation}

The following Lemma was proven in \cite[Lemma 5.1]{heilman20d}.  We reproduce that proof, since we require it below.

\begin{lemma}[\embolden{Key Lemma, $m=2$, Translations as Almost Eigenfunctions}]\label{treig}
Let $\Omega,\Omega^{c}\subset\R^{\adimn}$ maximize Problem \ref{prob2} for $m=2$.  Let $v\in\R^{\adimn}$.  Then
$$S(\langle v,N\rangle)(x)=\langle v,N(x)\rangle\frac{1}{\rho}\vnorm{\overline{\nabla} T_{\rho}1_{\Omega}(x)},\qquad\forall\,x\in\Sigma.$$
\end{lemma}
\begin{proof}
Since $T_{\rho}1_{\Omega}(x)$ is constant for all $x\in\partial\Omega$ by Lemma \ref{firstvarmaxns}, $\overline{\nabla} T_{\rho}1_{\Omega}(x)$ is parallel to $N(x)$ for all $x\in\partial\Omega$.  That is \eqref{nabeq2} holds, i.e.
\begin{equation}\label{firstve}
\overline{\nabla} T_{\rho}1_{\Omega}(x)=-N(x)\vnorm{\overline{\nabla} T_{\rho}1_{\Omega}(x)},\qquad\forall\,x\in\Sigma.
\end{equation}
From Definition \ref{oudef}, and then using the divergence theorem,
\begin{equation}\label{gre}
\begin{aligned}
\langle v,\overline{\nabla} T_{\rho}1_{\Omega}(x)\rangle
&=(1-\rho^{2})^{-(\adimn)/2}(2\pi)^{-(\adimn)/2}\Big\langle v,\int_{\Omega} \overline{\nabla}_{x}e^{-\frac{\vnorm{y-\rho x}^{2}}{2(1-\rho^{2})}}\,\d y\Big\rangle\\
&=(1-\rho^{2})^{-(\adimn)/2}(2\pi)^{-(\adimn)/2}\frac{\rho}{1-\rho^{2}}\int_{\Omega} \langle v,\,y-\rho x\rangle e^{-\frac{\vnorm{y-\rho x}^{2}}{2(1-\rho^{2})}}\,\d y\\
&=-(1-\rho^{2})^{-(\adimn)/2}(2\pi)^{-(\adimn)/2}\rho\int_{\Omega} \mathrm{div}_{y}\Big(ve^{-\frac{\vnorm{y-\rho x}^{2}}{2(1-\rho^{2})}}\Big)\,\d y\\
&=-(1-\rho^{2})^{-(\adimn)/2}(2\pi)^{-(\adimn)/2}\rho\int_{\Sigma}\langle v,N(y)\rangle e^{-\frac{\vnorm{y-\rho x}^{2}}{2(1-\rho^{2})}}\,\d y\\
&\stackrel{\eqref{sdef}}{=}-\rho\, S(\langle v,N\rangle)(x).
\end{aligned}
\end{equation}
Therefore,
\begin{equation}\label{lasteq}
\langle v,N(x)\rangle\vnorm{\overline{\nabla} T_{\rho}1_{\Omega}(x)}
\stackrel{\eqref{firstve}}{=}-\langle v,\overline{\nabla} T_{\rho}1_{\Omega}(x)\rangle\\
\stackrel{\eqref{gre}}{=}\rho\, S(\langle v,N\rangle)(x).
\end{equation}
\end{proof}

A priori finiteness of the above integrals was shown in \cite[Remark 5.2]{heilman20d} to follow from the divergence theorem.

\begin{lemma}[\embolden{Dilation as Almost Eigenfunction}]\label{mceig}
Let $\Omega,\Omega^{c}\subset\R^{\adimn}$ maximize Problem \ref{prob2} with $m=2$.  Then
%\begin{equation}\label{six1}
%S(H)(x)=\frac{1}{\rho^{2}}\Big(H(x)\vnorm{\overline{\nabla} T_{\rho}1_{\Omega}(x)}+\langle N(x),\overline{\nabla}\vnorm{\overline{\nabla} T_{\rho}1_{\Omega}(x)} \rangle\Big),\qquad\forall\,x\in\Sigma.
%\end{equation}
%More generally, if $M$ is an $(\adimn)\times(\adimn)$ matrix, then
%\begin{equation}\label{six1p}
%S(\mathrm{div}(MN))(x)=\frac{1}{\rho^{2}}\Big(\mathrm{div}(MN)\vnorm{\overline{\nabla} T_{\rho}1_{\Omega}(x)}+\langle MN(x),\overline{\nabla}\vnorm{\overline{\nabla} T_{\rho}1_{\Omega}(x)} \rangle\Big),\qquad\forall\,x\in\Sigma.
%\end{equation}%
%Lastly,
\begin{equation}\label{six1zv}
\begin{aligned}
&S(\langle\cdot ,N\rangle)(x)-\langle x,N(x)\rangle\vnormf{\overline{\nabla}T_{\rho}(1_{\Omega})(x)}\\
%&\qquad\qquad=\Big(\frac{1}{\rho^{2}}-1\Big)\Big(H(x)\vnorm{\overline{\nabla} T_{\rho}1_{\Omega}(x)}+\langle N(x),\overline{\nabla}\vnorm{\overline{\nabla} T_{\rho}1_{\Omega}(x)} \rangle\Big)\\
&\qquad=\Big(\frac{1}{\rho^{2}}-1\Big)\Big(\langle x,N(x)\rangle\vnorm{\overline{\nabla} T_{\rho}(1_{\Omega})(x)}
+\rho \frac{\d}{\d\rho}T_{\rho}(1_{\Omega})(x)\Big),\qquad\forall\,x\in\Sigma.
 \end{aligned}
\end{equation}
%\begin{equation}\label{six2}
%S(H-\frac{1}{1-\rho^{2}}\langle\cdot ,N\rangle)(x)=-\frac{1}{1-\rho^{2}}\langle x,N(x)\rangle\vnorm{\overline{\nabla} T_{\rho}1_{\Omega}(x)} ,\qquad\forall\,x\in\Sigma
%\end{equation}%
%\begin{equation}\label{six6}
%S\Big(H-\langle\cdot ,N\rangle\Big)(x)=\Big(H-\langle x,N(x)\rangle\Big)\vnorm{\overline{\nabla} T_{\rho}1_{\Omega}(x)}
%+\langle N(x),\overline{\nabla}\vnorm{\overline{\nabla} T_{\rho}1_{\Omega}(x)} \rangle ,\qquad\forall\,x\in\Sigma
%\end{equation}  % = S (H- <x,N>(1-rho^2)/(1-rho^2))= S(H-<x,N>(1/1-rho^2)) + rho^2 /(1-rho^2)S(<x,N>)
%  = S(H-<x,N>(1/1-rho^2)) + rho^2 SH+ (rho^2 /1-rho^2)<x,N>\nabla T_rho
\end{lemma}
\begin{proof}%
%Since $T_{\rho}1_{\Omega}(x)$ is constant for all $x\in\partial\Omega$ by Lemma \ref{firstvarmaxns}, $\overline{\nabla} T_{\rho}1_{\Omega}(x)$ is parallel to $N(x)$ for all $x\in\partial\Omega$.  That is \eqref{nabeq2} holds, i.e.
%\begin{equation}\label{firstvev}
%\overline{\nabla} T_{\rho}1_{\Omega}(x)=-N(x)\vnorm{\overline{\nabla} T_{\rho}1_{\Omega}(x)},\qquad\forall\,x\in\Sigma.
%\end{equation}
%Taking the divergence of \eqref{firstvev},
%\begin{equation}\label{firstvev2v}
%\mathrm{div}\overline{\nabla} T_{\rho}1_{\Omega}(x)
%=-\mathrm{div}(N(x))\vnorm{\overline{\nabla} T_{\rho}1_{\Omega}(x)}-\langle N(x),\overline{\nabla}\vnorm{\overline{\nabla} T_{\rho}1_{\Omega}(x)} \rangle.
%\end{equation}
Taking the gradient and divergence of \eqref{oudef},
\begin{equation}\label{six3v}
\begin{aligned}
&\mathrm{div}\overline{\nabla} T_{\rho}1_{\Omega}(x)
=(1-\rho^{2})^{-(\adimn)/2}(2\pi)^{-(\adimn)/2}\int_{\Omega}\mathrm{div}_{x}\overline{\nabla}_{x}e^{-\frac{\vnorm{y-\rho x}^{2}}{2(1-\rho^{2})}}\,\d y\\
&\qquad=(1-\rho^{2})^{-(\adimn)/2}(2\pi)^{-(\adimn)/2}\int_{\Omega}\mathrm{div}_{x}\Big(\rho\frac{ y-\rho x}{1-\rho^{2}} e^{-\frac{\vnorm{y-\rho x}^{2}}{2(1-\rho^{2})}}\Big)\,\d y\\
&\qquad=(1-\rho^{2})^{-(\adimn)/2}(2\pi)^{-(\adimn)/2}\int_{\Omega}\Big(\rho^{2}\frac{ \vnorm{y-\rho x}^{2}}{(1-\rho^{2})^{2}}-(\adimn)\frac{\rho^{2}}{1-\rho^{2}}\Big) e^{-\frac{\vnorm{y-\rho x}^{2}}{2(1-\rho^{2})}}\,\d y\\
&\qquad=(1-\rho^{2})^{-(\adimn)/2}(2\pi)^{-(\adimn)/2}\frac{-\rho^{2}}{1-\rho^{2}}\int_{\Omega}\mathrm{div}_{y}\Big((y-\rho x)e^{-\frac{\vnorm{y-\rho x}^{2}}{2(1-\rho^{2})}}\Big) \,\d y.
%&\qquad=(1-\rho^{2})^{-(\adimn)/2}(2\pi)^{-(\adimn)/2}\rho^{2}\int_{\R^{\adimn}}1_{\Omega}(y)\mathrm{div}_{y}\overline{\nabla}_{y}e^{-\frac{\vnorm{y-\rho x}^{2}}{2(1-\rho^{2})}} \,\d y
\end{aligned}
\end{equation}%
%Integrating by parts, i.e. using Green's identity in \eqref{six3}, and then Definition \ref{rbdef},
%\begin{equation}\label{grev2}  %  div(fM\nabla g - gM\nabla f) =  fdiv (M\nabla g) - g div(M\nabla f)
%\begin{aligned}
%\mathrm{div}M\overline{\nabla} T_{\rho}1_{\Omega}(x)
%&=(1-\rho^{2})^{-(\adimn)/2}(2\pi)^{-(\adimn)/2}\rho^{2}\int_{\R^{\adimn}}\Big(\mathrm{div}_{y}M\overline{\nabla}_{y}1_{\Omega}(y)\Big)e^{-\frac{\vnorm{y-\rho x}^{2}}{2(1-\rho^{2})}} \,\d y\\
%&=(1-\rho^{2})^{-(\adimn)/2}(2\pi)^{-(\adimn)/2}\rho^{2}\int_{\Sigma}\mathrm{div}_{y}(-MN(y))e^{-\frac{\vnorm{y-\rho x}^{2}}{2(1-\rho^{2})}} \,\d y\\
%%&=(1-\rho^{2})^{-(\adimn)/2}(2\pi)^{-(\adimn)/2}\rho^{2}\int_{\Sigma}(-(y))e^{-\frac{\vnorm{y-\rho x}^{2}}{2(1-\rho^{2})}} \,\d y
%&\stackrel{\eqref{sdef}}{=}-\rho^{2}\, S(\mathrm{div}(MN))(x).
%\end{aligned}
%\end{equation}
%\snote{missing lower dimensional term?}  Then \eqref{grev2} and \eqref{firstvev2} prove  \eqref{six1}.  Alternatively,
Applying the divergence theorem to the last equality in \eqref{six3v},
\begin{equation}\label{six3tv}
\begin{aligned}
\mathrm{div}\overline{\nabla} T_{\rho}1_{\Omega}(x)
&=(1-\rho^{2})^{-(\adimn)/2}(2\pi)^{-(\adimn)/2}\frac{-\rho^{2}}{1-\rho^{2}}\int_{\Sigma}\Big\langle (y-\rho x), N(y)\Big\rangle e^{-\frac{\vnorm{y-\rho x}^{2}}{2(1-\rho^{2})}} \,\d y\\
&\stackrel{\eqref{sdef}}{=}-\frac{\rho^{2}}{1-\rho^{2}}\Big(S(\langle \cdot, N\rangle)(x)  -\rho \langle x,S(N)(x)\rangle\Big)\\
&\stackrel{\eqref{lasteq}}{=}-\frac{\rho^{2}}{1-\rho^{2}}\Big(S(\langle \cdot, N\rangle)(x)  -\langle x,N(x)\rangle\vnorm{\nabla T_{\rho}1_{\Omega}(x)}\Big).
\end{aligned}
\end{equation}
This equation and \eqref{firstve} proves \eqref{six1zv}, together with
\begin{flalign*}
\mathrm{div}\overline{\nabla} T_{\rho}1_{\Omega}(x)
&=\overline{\Delta}T_{\rho}1_{\Omega}(x)-\langle x,\nabla T_{\rho}1_{\Omega}(x)+\langle x,\overline{\nabla} T_{\rho}1_{\Omega}(x)\rangle\\
&\stackrel{\eqref{oup}\wedge\eqref{firstve}}{=}-\rho\frac{d}{d\rho}T_{\rho}1_{\Omega}(x)-\langle x,N(x)\rangle\vnormf{\overline{\nabla} T_{\rho}1_{\Omega}(x)}.
\end{flalign*}
\end{proof}

A priori finiteness of the above integrals will be shown in Remark \ref{drk} below.

\subsection{More than Two Sets}

Let $v\in\R^{\adimn}$ and denote $f_{ij}\colonequals\langle v,N_{ij}\rangle$ for all $1\leq i,j\leq m$.  For simplicity of notation in the formulas below, if $1\leq i\leq m$ and if a vector $N(x)$ appears inside an integral over $\partial\Omega_{i}$, then $N(x)$ denotes the unit exterior pointing normal vector to $\Omega_{i}$ at $x\in\redb\Omega_{i}$.  Similarly, for simplicity of notation, we denote $\langle v,N\rangle$ as the collection of functions $(\langle v,N_{ij}\rangle)_{1\leq i<j\leq m}$.  For any $1\leq i<j\leq m$, define
\begin{equation}\label{sdef2}
S_{ij}(\langle v,N\rangle)(x)\colonequals (1-\rho^{2})^{-(\adimn)/2}(2\pi)^{-(\adimn)/2}\Big(\int_{\partial\Omega_{i}}-\int_{\partial\Omega_{j}}\Big)\langle v,N(y)\rangle e^{-\frac{\vnorm{y-\rho x}^{2}}{2(1-\rho^{2})}}\,\d y,\,\forall\,x\in\Sigma_{ij}.
\end{equation}

\begin{lemma}[\embolden{Key Lemma, $m\geq 2$, Translations as Almost Eigenfunctions}, {\cite[Lemma 5.4]{heilman20d}}]\label{treig2}
Let $\Omega_{1},\ldots,\Omega_{m}\subset\R^{\adimn}$ maximize Problem \ref{prob2}.  Fix $1\leq i<j\leq m$.  Let $v\in\R^{\adimn}$.  Then
$$S_{ij}(\langle v,N\rangle)(x)=\langle v,N_{ij}(x)\rangle\frac{1}{\rho}\vnorm{\overline{\nabla} T_{\rho}(1_{\Omega_{i}}-1_{\Omega_{j}})(x)},\qquad\forall\,x\in\Sigma_{ij}.$$
\end{lemma}%

\begin{lemma}[\embolden{Second Variation of Translations, Multiple Sets}, {\cite[Lemma 5.5]{heilman20d}}]\label{keylem}
Let $v\in\R^{\adimn}$.  Let $\Omega_{1},\ldots,\Omega_{m}\subset\R^{\adimn}$ maximize Problem \ref{prob2}.  For each $1\leq i\leq m$, let $\{\Omega_{i}^{(s)}\}_{s\in(-1,1)}$ be the variation of $\Omega_{i}$ corresponding to the constant vector field $X\colonequals v$.  Assume that
$$\int_{\partial\Omega_{i}}\langle v,N(x)\rangle \gamma_{\adimn}(x)\,\d x=0,\qquad\forall\,1\leq i\leq m.$$
Then
\begin{flalign*}
&\frac{1}{2}\frac{\d^{2}}{\d s^{2}}\Big|_{s=0}\sum_{i=1}^{m}\int_{\R^{\adimn}}1_{\Omega_{i}^{(s)}}(x)T_{\rho}1_{\Omega_{i}^{(s)}}(x)\gamma_{\adimn}(x)\,\d x\\
&\qquad\qquad\qquad=\Big(\frac{1}{\rho}-1\Big)\sum_{1\leq i<j\leq m}\int_{\Sigma_{ij}}\vnormf{\overline{\nabla}T_{\rho}(1_{\Omega_{i}}-1_{\Omega_{j}})(x)}\langle v,N_{ij}(x)\rangle^{2}\gamma_{\adimn}(x)\,\d x.
\end{flalign*}%
\end{lemma}

\begin{lemma}[\embolden{Dilation as Almost Eigenfunction}]\label{mceigp}
Let $\Omega_{1},\ldots,\Omega_{m}\subset\R^{\adimn}$ maximize Problem \ref{prob2}.  Then for all $1\leq i,j\leq m$,
%\begin{equation}\label{six1}
%S(H)(x)=\frac{1}{\rho^{2}}\Big(H(x)\vnorm{\overline{\nabla} T_{\rho}1_{\Omega}(x)}+\langle N(x),\overline{\nabla}\vnorm{\overline{\nabla} T_{\rho}1_{\Omega}(x)} \rangle\Big),\qquad\forall\,x\in\Sigma.
%\end{equation}
%More generally, if $M$ is an $(\adimn)\times(\adimn)$ matrix, then
%\begin{equation}\label{six1p}
%S(\mathrm{div}(MN))(x)=\frac{1}{\rho^{2}}\Big(\mathrm{div}(MN)\vnorm{\overline{\nabla} T_{\rho}1_{\Omega}(x)}+\langle MN(x),\overline{\nabla}\vnorm{\overline{\nabla} T_{\rho}1_{\Omega}(x)} \rangle\Big),\qquad\forall\,x\in\Sigma.
%\end{equation}%
%Lastly,
\begin{equation}\label{six1z}
\begin{aligned}
&S_{ij}(\langle\cdot ,N\rangle)(x)-\langle x,N_{ij}(x)\rangle\vnormf{\overline{\nabla}T_{\rho}(1_{\Omega_{i}}-1_{\Omega_{j}})(x)}\\
%&\qquad\qquad=\Big(\frac{1}{\rho^{2}}-1\Big)\Big(H(x)\vnorm{\overline{\nabla} T_{\rho}1_{\Omega}(x)}+\langle N(x),\overline{\nabla}\vnorm{\overline{\nabla} T_{\rho}1_{\Omega}(x)} \rangle\Big)\\
&\qquad=\Big(\frac{1}{\rho^{2}}-1\Big)\Big(\langle x,N_{ij}(x)\rangle\vnorm{\overline{\nabla} T_{\rho}(1_{\Omega_{i}}-1_{\Omega_{j}})(x)}
+\rho \frac{\d}{\d\rho}T_{\rho}(1_{\Omega_{i}}-1_{\Omega_{j}})(x)\Big),\qquad\forall\,x\in\Sigma_{ij}.
 \end{aligned}
\end{equation}
%\begin{equation}\label{six2}
%S(H-\frac{1}{1-\rho^{2}}\langle\cdot ,N\rangle)(x)=-\frac{1}{1-\rho^{2}}\langle x,N(x)\rangle\vnorm{\overline{\nabla} T_{\rho}1_{\Omega}(x)} ,\qquad\forall\,x\in\Sigma
%\end{equation}%
%\begin{equation}\label{six6}
%S\Big(H-\langle\cdot ,N\rangle\Big)(x)=\Big(H-\langle x,N(x)\rangle\Big)\vnorm{\overline{\nabla} T_{\rho}1_{\Omega}(x)}
%+\langle N(x),\overline{\nabla}\vnorm{\overline{\nabla} T_{\rho}1_{\Omega}(x)} \rangle ,\qquad\forall\,x\in\Sigma
%\end{equation}  % = S (H- <x,N>(1-rho^2)/(1-rho^2))= S(H-<x,N>(1/1-rho^2)) + rho^2 /(1-rho^2)S(<x,N>)
%  = S(H-<x,N>(1/1-rho^2)) + rho^2 SH+ (rho^2 /1-rho^2)<x,N>\nabla T_rho
\end{lemma}

\begin{proof}
From \eqref{nabeq3}, for all $1\leq i,j\leq m$,
\begin{equation}\label{firstvep}
\overline{\nabla}T_{\rho}(1_{\Omega_{i}}-1_{\Omega_{j}})(x)=-N_{ij}(x)\vnorm{\overline{\nabla}T_{\rho}(1_{\Omega_{i}}-1_{\Omega_{j}})(x)},\qquad\forall\,x\in\Sigma_{ij}.
\end{equation}
Taking the divergence of \eqref{firstvep}, for all $1\leq i,j\leq m$,
\begin{equation}\label{firstvev2}
\begin{aligned}
\mathrm{div}\overline{\nabla} T_{\rho}(1_{\Omega_{i}}-1_{\Omega_{j}})(x)
&=-\mathrm{div}(N_{ij}(x))\vnorm{\overline{\nabla}T_{\rho}(1_{\Omega_{i}}-1_{\Omega_{j}})(x)}\\
&\qquad\qquad\qquad-\langle N_{ij}(x),\overline{\nabla}\vnorm{\overline{\nabla} T_{\rho}(1_{\Omega_{i}}-1_{\Omega_{j}})(x)} \rangle.
\end{aligned}
\end{equation}
Taking the gradient and divergence of \eqref{oudef},
\begin{equation}\label{six3}
\begin{aligned}
&\mathrm{div}\overline{\nabla} T_{\rho}(1_{\Omega_{i}}-1_{\Omega_{j}})(x)
=(1-\rho^{2})^{-(\adimn)/2}(2\pi)^{-(\adimn)/2}\Big(\int_{\Omega_{i}}-\int_{\Omega_{j}}\Big)\mathrm{div}_{x}\overline{\nabla}_{x}e^{-\frac{\vnorm{y-\rho x}^{2}}{2(1-\rho^{2})}}\,\d y\\
&\qquad=(1-\rho^{2})^{-(\adimn)/2}(2\pi)^{-(\adimn)/2}\Big(\int_{\Omega_{i}}-\int_{\Omega_{j}}\Big)\mathrm{div}_{x}\Big(\rho\frac{ y-\rho x}{1-\rho^{2}} e^{-\frac{\vnorm{y-\rho x}^{2}}{2(1-\rho^{2})}}\Big)\,\d y\\
&\qquad=(1-\rho^{2})^{-(\adimn)/2}(2\pi)^{-(\adimn)/2}\Big(\int_{\Omega_{i}}-\int_{\Omega_{j}}\Big)\Big(\rho^{2}\frac{ \vnorm{y-\rho x}^{2}}{(1-\rho^{2})^{2}}-(\adimn)\frac{\rho^{2}}{1-\rho^{2}}\Big) e^{-\frac{\vnorm{y-\rho x}^{2}}{2(1-\rho^{2})}}\,\d y\\
&\qquad=(1-\rho^{2})^{-(\adimn)/2}(2\pi)^{-(\adimn)/2}\frac{-\rho^{2}}{1-\rho^{2}}\Big(\int_{\Omega_{i}}-\int_{\Omega_{j}}\Big)\mathrm{div}_{y}\Big((y-\rho x)e^{-\frac{\vnorm{y-\rho x}^{2}}{2(1-\rho^{2})}}\Big) \,\d y.
%&\qquad=(1-\rho^{2})^{-(\adimn)/2}(2\pi)^{-(\adimn)/2}\rho^{2}\int_{\R^{\adimn}}1_{\Omega}(y)\mathrm{div}_{y}\overline{\nabla}_{y}e^{-\frac{\vnorm{y-\rho x}^{2}}{2(1-\rho^{2})}} \,\d y
\end{aligned}
\end{equation}%
%Integrating by parts, i.e. using Green's identity in \eqref{six3}, and then Definition \ref{rbdef},
%\begin{equation}\label{grev2}  %  div(fM\nabla g - gM\nabla f) =  fdiv (M\nabla g) - g div(M\nabla f)
%\begin{aligned}
%\mathrm{div}M\overline{\nabla} T_{\rho}1_{\Omega}(x)
%&=(1-\rho^{2})^{-(\adimn)/2}(2\pi)^{-(\adimn)/2}\rho^{2}\int_{\R^{\adimn}}\Big(\mathrm{div}_{y}M\overline{\nabla}_{y}1_{\Omega}(y)\Big)e^{-\frac{\vnorm{y-\rho x}^{2}}{2(1-\rho^{2})}} \,\d y\\
%&=(1-\rho^{2})^{-(\adimn)/2}(2\pi)^{-(\adimn)/2}\rho^{2}\int_{\Sigma}\mathrm{div}_{y}(-MN(y))e^{-\frac{\vnorm{y-\rho x}^{2}}{2(1-\rho^{2})}} \,\d y\\
%%&=(1-\rho^{2})^{-(\adimn)/2}(2\pi)^{-(\adimn)/2}\rho^{2}\int_{\Sigma}(-(y))e^{-\frac{\vnorm{y-\rho x}^{2}}{2(1-\rho^{2})}} \,\d y
%&\stackrel{\eqref{sdef}}{=}-\rho^{2}\, S(\mathrm{div}(MN))(x).
%\end{aligned}
%\end{equation}
%\snote{missing lower dimensional term?}  Then \eqref{grev2} and \eqref{firstvev2} prove  \eqref{six1}.  Alternatively,
Applying the divergence theorem to the last equality in \eqref{six3},
\begin{equation}\label{six3t}
\begin{aligned}
&\mathrm{div}\overline{\nabla} T_{\rho}(1_{\Omega_{i}}-1_{\Omega_{j}})(x)\\
&\qquad=(1-\rho^{2})^{-(\adimn)/2}(2\pi)^{-(\adimn)/2}\frac{-\rho^{2}}{1-\rho^{2}}\Big(\int_{\partial\Omega_{i}}-\int_{\partial\Omega_{j}}\Big)\Big\langle (y-\rho x), N(y)\Big\rangle e^{-\frac{\vnorm{y-\rho x}^{2}}{2(1-\rho^{2})}} \,\d y\\
&\qquad\stackrel{\eqref{sdef2}}{=}-\frac{\rho^{2}}{1-\rho^{2}}\Big(S_{ij}(\langle \cdot, N\rangle)(x)  -\rho \langle x,S_{ij}(N)(x)\rangle\Big)\\
&\qquad\stackrel{\eqref{firstvep}\wedge \eqref{gre}}{=}-\frac{\rho^{2}}{1-\rho^{2}}\Big(S_{ij}(\langle \cdot, N\rangle)(x)  -\langle x,N_{ij}(x)\rangle\vnorm{\overline{\nabla} T_{\rho}(1_{\Omega_{i}}-1_{\Omega_{j}})}\Big).
\end{aligned}
\end{equation}
This equation and \eqref{firstvev2} proves \eqref{six1z}, together with
\begin{flalign*}
&\mathrm{div}\overline{\nabla} T_{\rho}(1_{\Omega_{i}}-1_{\Omega_{j}})(x)\\
&\qquad\qquad=\overline{\Delta}T_{\rho}(1_{\Omega_{i}}-1_{\Omega_{j}})(x)
-\langle x,\nabla T_{\rho}(1_{\Omega_{i}}-1_{\Omega_{j}})(x)\rangle
+\langle x,\overline{\nabla} T_{\rho}(1_{\Omega_{i}}-1_{\Omega_{j}})(x)\rangle\\
&\qquad\qquad\stackrel{\eqref{oup}\wedge\eqref{firstvep}}{=}-\rho\frac{d}{d\rho}T_{\rho}(1_{\Omega_{i}}-1_{\Omega_{j}})(x)-\langle x,N_{ij}(x)\rangle\vnormf{\overline{\nabla} T_{\rho}(1_{\Omega_{i}}-1_{\Omega_{j}})(x)}.
\end{flalign*}

\end{proof}

\begin{remark}\label{drk}
To justify the use of the divergence theorem in \eqref{six3t}, let $r>0$ and note that we can differentiate under the integral sign of  $T_{\rho}1_{\Omega\cap B(0,r)}(x)$ to get
\begin{equation}\label{grep}
\begin{aligned}
\overline{\nabla} T_{\rho}1_{\Omega\cap B(0,r)}(x)
&=(1-\rho^{2})^{-(\adimn)/2}(2\pi)^{-(\adimn)/2}\int_{\Omega\cap B(0,r)} \overline{\nabla}_{x}e^{-\frac{\vnorm{y-\rho x}^{2}}{2(1-\rho^{2})}}\,\d y\\
&=(1-\rho^{2})^{-(\adimn)/2}(2\pi)^{-(\adimn)/2}\frac{\rho}{1-\rho^{2}}\int_{\Omega\cap B(0,r)}(y-\rho x) e^{-\frac{\vnorm{y-\rho x}^{2}}{2(1-\rho^{2})}}\,\d y\\
%&=-(1-\rho^{2})^{-(\adimn)/2}(2\pi)^{-(\adimn)/2}\rho\int_{\Omega\cap B(0,r)} \mathrm{div}_{y}\Big(ve^{-\frac{\vnorm{y-\rho x}^{2}}{2(1-\rho^{2})}}\Big)\,\d y\\
%&=-(1-\rho^{2})^{-(\adimn)/2}(2\pi)^{-(\adimn)/2}\rho\int_{(\Sigma\cap B(0,r))\cup(\Omega\cap\partial B(0,r))}\langle v,N(y)\rangle e^{-\frac{\vnorm{y-\rho x}^{2}}{2(1-\rho^{2})}}\,\d y.
\end{aligned}
\end{equation}
\begin{equation}\label{grepv2}
\begin{aligned}
&\mathrm{div}\overline{\nabla} T_{\rho}1_{\Omega\cap B(0,r)}(x)\\
&\quad=(1-\rho^{2})^{-(\adimn)/2}(2\pi)^{-(\adimn)/2}\int_{\Omega\cap B(0,r)}
\Big(\rho^{2}\frac{ \vnorm{y-\rho x}^{2}}{(1-\rho^{2})^{2}}-(\adimn)\frac{\rho^{2}}{1-\rho^{2}}\Big) e^{-\frac{\vnorm{y-\rho x}^{2}}{2(1-\rho^{2})}}\,\d y\\
&\quad=(1-\rho^{2})^{-(\adimn)/2}(2\pi)^{-(\adimn)/2}\int_{\Omega\cap B(0,r)}
\mathrm{div}_{y}\Big((y-\rho x)e^{-\frac{\vnorm{y-\rho x}^{2}}{2(1-\rho^{2})}}\Big) \,\d y\\
&\quad=(1-\rho^{2})^{-(\adimn)/2}(2\pi)^{-(\adimn)/2}\int_{[\Sigma\cap B(0,r)]\cup[\Omega\cap\partial B(0,r)]}
\langle (y-\rho x),N(y)\rangle e^{-\frac{\vnorm{y-\rho x}^{2}}{2(1-\rho^{2})}}\Big) \,\d y.
\end{aligned}
\end{equation}

Fix $r'>0$.  Fix $x\in\R^{\adimn}$ with $\vnorm{x}<r'$.  The last integral in \eqref{grep} over $\Omega\cap \partial B(0,r)$ goes to zero as $r\to\infty$ uniformly over all such $\vnorm{x}<r'$.  Also
$\mathrm{div}\overline{\nabla} T_{\rho}1_{\Omega}(x)$
exists a priori for all $x\in\R^{\adimn}$, while
\begin{flalign*}
&\abs{\mathrm{div}\overline{\nabla} T_{\rho}1_{\Omega}(x)-\mathrm{div}\overline{\nabla} T_{\rho}1_{\Omega\cap B(0,r)}(x)}\\
&\stackrel{\eqref{grepv2}}{=}(1-\rho^{2})^{-(\adimn)/2}(2\pi)^{-(\adimn)/2}\abs{\int_{\Omega\cap B(0,r)^{c}}
\Big(\rho^{2}\frac{ \vnorm{y-\rho x}^{2}}{(1-\rho^{2})^{2}}-(\adimn)\frac{\rho^{2}}{1-\rho^{2}}\Big) e^{-\frac{\vnorm{y-\rho x}^{2}}{2(1-\rho^{2})}}\,\d y}\\
&\qquad
\leq (1-\rho^{2})^{-(\adimn)/2}(2\pi)^{-(\adimn)/2}\int_{B(0,r)^{c}}
\Big(\rho^{2}\frac{ \vnorm{y-\rho x}^{2}}{(1-\rho^{2})^{2}}+(\adimn)\frac{\rho^{2}}{1-\rho^{2}}\Big) e^{-\frac{\vnorm{y-\rho x}^{2}}{2(1-\rho^{2})}}\,\d y.
\end{flalign*}
And the last integral goes to zero as $r\to\infty$, uniformly over all $\vnorm{x}<r'$.
\end{remark}

\section{Symmetric Sets}

\begin{proof}[Proof of Theorem \ref{symthm}]
Consider the functions $f,g\colon\Sigma\to\R$ defined by
$$g(x)\colonequals \langle x,N(x)\rangle,\qquad\forall\,x\in\Sigma.$$
$$f(x)\colonequals (x_{\adimn}^{2}-1)g(x),\qquad\forall\,x\in\Sigma.$$
Note that $\langle x,N(x)\rangle$ is constant as $x_{\adimn}$ varies, and $(x_{\adimn}^{2}-1)$ is constant as $x_{1},\ldots,x_{\sdimn}$ varies.  We can therefore apply Fubini's Theorem when we compute $S(f)$, and also note that $\int_{\Sigma}f(x)\gamma_{\adimn}(x)\,\d x=0$ since $\int_{\R}(x_{\adimn}^{2}-1)\gamma_{1}(x_{\adimn})\,\d x_{\adimn}=0$.  From \eqref{sdef}, and using the well-known property of Hermite polynomials on the real line that $T_{\rho}(x_{1}^{2}-1)=\rho^{2}(x_{1}^{2}-1)$ for all $x_{1}\in\R$, we have
$$S(f)(x)=\rho^{2}(x_{\adimn}^{2}-1) S(g)(x),\qquad\forall\,x\in\Sigma.$$
From Lemma \ref{mceig},
\begin{equation}\label{seven1}
S(g)(x)=\frac{1}{\rho^{2}}g(x)\vnormf{\overline{\nabla}T_{\rho}1_{\Omega}(x)}+\rho(1-\rho^{2})\frac{\d}{\d\rho}T_{\rho}(1_{\Omega})(x),\qquad\forall\,x\in\Sigma.
\end{equation}%
Plugging this into Lemma \ref{lemma7p}, and using also $\int_{\R}(x_{\adimn}^{2}-1)^{2}\gamma_{1}(x_{\adimn})\,\d x_{\adimn}=1$ and that $\vnormf{\overline{\nabla}T_{\rho}1_{\Omega}(x)}$ is constant as $x_{\adimn}$ varies, the variation of $\{\Omega^{(s)}\}_{s\in(-1,1)}$ corresponding to this choice of $f$ satisfies
\begin{equation}\label{seven2}
\begin{aligned}
&\frac{1}{2}\frac{\d^{2}}{\d s^{2}}\Big|_{s=0}\int_{\R^{\adimn}}1_{\Omega^{(s)}}(x)T_{\rho}1_{\Omega^{(s)}}(x)\gamma_{\adimn}(x)\,\d x\\
&\qquad=\int_{\Sigma}\Big(f(x)S(f)(x)-\abs{f(x)}^{2}\vnormf{\overline{\nabla}T_{\rho}1_{\Omega}(x)})\Big)\gamma_{\adimn}(x)\,\d x\\
&\qquad=\int_{\Sigma}\Big(\rho^{2}g(x)S(g)(x)-\abs{g(x)}^{2}\vnormf{\overline{\nabla}T_{\rho}1_{\Omega}(x)})\Big)\gamma_{\adimn}(x)\,\d x\\
&\qquad\stackrel{\eqref{seven1}}{=}\rho^{3}(1-\rho^{2})\int_{\Sigma}\frac{\d}{\d\rho}T_{\rho}(1_{\Omega})(x)\langle x,N(x)\rangle\gamma_{\adimn}(x)\,\d x\geq0.
\end{aligned}
\end{equation}
The last inequality follows by assumption.  Now, consider the function
$$\widetilde{f}(x)\colonequals (x_{\adimn}^{2}-1)\abs{\langle x,N(x)\rangle},\qquad\forall\,x\in\Sigma.$$
If $\Omega$ is not star shaped, then since $G(x,y)>0$ for all $x,y\in\R^{\adimn}$, we have a strict inequality
\begin{equation}\label{seven3}
\begin{aligned}
\int_{\Sigma}\abs{g(x)}S(\abs{g})(x)\gamma_{\adimn}(x)\,\d x
&\stackrel{\eqref{sdef}}{=}\int_{\Sigma}\int_{\Sigma}\abs{g(x)} G(x,y)\abs{g(y)} \, \d x \d y\\
&>\int_{\Sigma}\int_{\Sigma}g(x) G(x,y)g(y) \, \d x \d y
\stackrel{\eqref{sdef}}{=}\int_{\Sigma}gS(g)(x)\gamma_{\adimn}(x)\,\d x.
\end{aligned}
\end{equation}
So, the variation of $\{\Omega^{(s)}\}_{s\in(-1,1)}$ corresponding to this choice of $\widetilde{f}$ satisfies,
\begin{flalign*}
&\frac{1}{2}\frac{\d^{2}}{\d s^{2}}\Big|_{s=0}\int_{\R^{\adimn}}1_{\Omega^{(s)}}(x)T_{\rho}1_{\Omega^{(s)}}(x)\gamma_{\adimn}(x)\,\d x\\
&\qquad=\int_{\Sigma}\Big(\widetilde{f}(x)S(\widetilde{f})(x)-\absf{\widetilde{f}(x)}^{2}\vnormf{\overline{\nabla}T_{\rho}1_{\Omega}(x)}\Big)\gamma_{\adimn}(x)\,\d x\\
&\qquad=\int_{\Sigma}\Big(\rho^{2}\abs{g(x)}S(\abs{g})(x)-\abs{g(x)}^{2}\vnormf{\overline{\nabla}T_{\rho}1_{\Omega}(x)}\Big)\gamma_{\adimn}(x)\,\d x\\
&\qquad\stackrel{\eqref{seven3}}{>}\int_{\Sigma}\Big(\rho^{2}g(x)S(g)(x)-\abs{g(x)}^{2}\vnormf{\overline{\nabla}T_{\rho}1_{\Omega}(x)}\Big)\gamma_{\adimn}(x)\,\d x
\stackrel{\eqref{seven2}}{\geq}0.
\end{flalign*}
Since we have found a variation $\widetilde{f}$ with positive second derivative, $\int_{\Sigma}\widetilde{f}(x)\gamma_{\adimn}(x)\,\d x=0$, and $\widetilde{f}(-x)=\widetilde{f}(x)$ for all $x\in\Sigma$, we have arrived at a contradiction.  We conclude that $\Omega$ or $\Omega^{c}$ is star-shaped.
\end{proof}
%
%
%\subsection{Bilinear Case ?}
%
%Motivated by \cite{chakrabarti10}
%
%\snote{add discussion about the bilinear case to the introduction}
%
%\snote{Re-do all above for bilinear case?  Maybe comment out proofs?  Use result below or save for after results.  Use Lemma \ref{mceign}.}

%
\section{Negative Correlation, Dimension Reduction}\label{negsec}

In this section, we consider the case that $\rho<0$ in Problem \ref{prob2}.  When $\rho<0$ and $h\colon\R^{\adimn}\to[-1,1]$ is measurable, then quantity
$$\int_{\R^{\adimn}}h(x)T_{\rho}h(x)\gamma_{\adimn}(x)\,\d x$$
could be negative, so a few parts of the above argument do not work, namely the existence Lemma \ref{existlem}.  We therefore replace the noise stability with a more general bilinear expression, guaranteeing existence of the corresponding problem.  The remaining parts of the argument are essentially identical, mutatis mutandis.  We indicate below where the arguments differ in the bilinear case.

When $\rho<0$, we look for a minimum of noise stability, rather than a maximum.  Correspondingly, we expect that the plurality function minimizes noise stability when $\rho<0$.  If $\rho<0$, then \eqref{oudef} implies that
$$\int_{\R^{\adimn}}h(x)T_{\rho}h(x)\gamma_{\adimn}(x)\,\d x
=\int_{\R^{\adimn}}h(x)T_{(-\rho)}h(-x)\gamma_{\adimn}(x)\,\d x.$$
So, in order to understand the minimum of noise stability for negative correlations, it suffices to consider the following bilinear version of the standard simplex problem with positive correlation.

\begin{prob}[\embolden{Standard Simplex Problem, Bilinear Version, Positive Correlation, Equal Measure Constraint}]\label{prob2n}
Let $m\geq3$.  Fix $0<\rho<1$.  Find measurable sets $\Omega_{1},\ldots\Omega_{m},\Omega_{1}',\ldots\Omega_{m}'\subset\R^{\adimn}$ with $\cup_{i=1}^{m}\Omega_{i}=\cup_{i=1}^{m}\Omega_{i}'=\R^{\adimn}$ and $\gamma_{\adimn}(\Omega_{i})=\gamma_{\adimn}(\Omega_{i}')$ for all $1\leq i\leq m$ that minimize
$$\sum_{i=1}^{m}\int_{\R^{\adimn}}1_{\Omega_{i}}(x)T_{\rho}1_{\Omega_{i}'}(x)\gamma_{\adimn}(x)\,\d x,$$
subject to the above constraints.
\end{prob}

We remark that even the case $m=2$ of this bilinear problem is interesting, as it played a role in \cite{chakrabarti10} in proving an optimal lower bound for the Gap-Hamming-Distance problem in communication complexity.  However, the problem relevant to \cite{chakrabarti10} also requires an extra symmetry assumption.

\begin{conj}[\embolden{Standard Simplex Conjecture, Bilinear Version, Positive Correlation} {\cite{isaksson11}}]\label{conj2n}
Let $\Omega_{1},\ldots\Omega_{m},\Omega_{1}',\ldots\Omega_{m}'\subset\R^{\adimn}$ minimize Problem \ref{prob2}.  Assume that $m-1\leq\adimn$.  Fix $0<\rho<1$.  Let $z_{1},\ldots,z_{m}\in\R^{\adimn}$ be the vertices of a regular simplex in $\R^{\adimn}$ centered at the origin.  Then, for all $1\leq i\leq m$,
$$\Omega_{i}=-\Omega_{i}'=\{x\in\R^{\adimn}\colon\langle x,z_{i}\rangle=\max_{1\leq j\leq m}\langle x,z_{j}\rangle\}.$$
\end{conj}
%It is known that Conjecture \ref{conj2n} is false when $(a_{1},\ldots,a_{m})\neq(1/m,\ldots,1/m)$ \cite{heilman14}.
%In the case that $a_{i}=1/m$ for all $1\leq i\leq m$, it is assumed that $w=0$ in Conjecture \ref{conj2n}.

Since we consider a bilinear version of noise stability in Problem \ref{prob2n}, existence of an optimizer is easier than in Problem \ref{prob2}.

\begin{lemma}[\embolden{Existence of a Minimizer}]\label{existlemn}
Let $0<\rho<1$ and let $m\geq2$.  Then there exist measurable sets $\Omega_{1},\ldots\Omega_{m},\Omega_{1}',\ldots\Omega_{m}'$ that minimize Problem \ref{prob2n}.
\end{lemma}
\begin{proof}
Define $\Delta_{m}$ as in \eqref{deltadef}.  Let $f,g\colon\R^{\adimn}\to\Delta_{m}$.  The set $D_{0}\colonequals\{(f,g)\colon\quad f,g\colon\R^{\adimn}\to\Delta_{m},\, \int_{\R^{\adimn}}f(x)\gamma_{\adimn}\,\d x=\int_{\R^{\adimn}}g(x)\gamma_{\adimn}\,\d x\}$ is norm closed, bounded, and convex, therefore it is weakly compact and convex.  Consider the function
$$C(f,g)\colonequals\sum_{i=1}^{m}\int_{\R^{\adimn}}f_{i}(x)T_{\rho}g_{i}(x)\gamma_{\adimn}(x)\,\d x.$$
This function is weakly continuous on $D_{0}$, and $D_{0}$ is weakly compact, so there exists $(\widetilde{f},\widetilde{g})\in D_{0}$ such that $C(\widetilde{f},\widetilde{g})=\min_{f,g\in D_{0}}C(f,g)$. Since $C$ is bilinear and $D_{0}$ is convex, the minimum of $C$ must be achieved at an extreme point of $D_{0}$.  (If e.g. $g$ is fixed, then $f\mapsto C(f,g)$ is linear in $f$, and $\{f\colon(f,g)\in D_{0}\}$ is a convex set.)  Let $e_{1},\ldots,e_{m}$ denote the standard basis of $\R^{m}$, so that $f,g$ takes their values in $\{e_{1},\ldots,e_{m}\}$. Then, for any $1\leq i\leq m$, define $\Omega_{i}\colonequals\{x\in\R^{\adimn}\colon f(x)=e_{i}\}$ and $\Omega_{i}'\colonequals\{x\in\R^{\adimn}\colon g(x)=e_{i}\}$.  Note that $f_{i}=1_{\Omega_{i}}$ and $g_{i}=1_{\Omega_{i}'}$ for all $1\leq i\leq m$.
\end{proof}

\begin{lemma}[\embolden{Regularity of a Minimizer}]\label{reglemn}
Let $\Omega_{1},\ldots,\Omega_{m},\Omega_{1}',\ldots,\Omega_{m}'\subset\R^{\adimn}$ be the measurable sets minimizing Problem \ref{prob2}, guaranteed to exist by Lemma \ref{existlemn}.  Then the sets $\Omega_{1},\ldots,\Omega_{m},\Omega_{1}',\ldots,\Omega_{m}'$ have locally finite surface area.  Moreover, for all $1\leq i\leq m$ and for all $x\in\partial\Omega_{i}$, there exists a neighborhood $U$ of $x$ such that $U\cap \partial\Omega_{i}$ is a finite union of $C^{\infty}$ $\sdimn$-dimensional manifolds.  The same holds for $\Omega_{1}',\ldots,\Omega_{m}'$.
\end{lemma}

We denote $\Sigma_{ij}\colonequals(\redb\Omega_{i})\cap(\redb\Omega_{j}), \Sigma_{ij}'\colonequals(\redb\Omega_{i}')\cap(\redb\Omega_{j}')$ for all $1\leq i<j\leq m$.

\begin{lemma}[\embolden{The First Variation for Minimizers}]\label{firstvarmaxnsn}
Suppose $\Omega_{1},\ldots,\Omega_{m},\Omega_{1}',\ldots,\Omega_{m}'\subset\R^{\adimn}$ minimize Problem \ref{prob2n}.  Then for all $1\leq i<j\leq m$, there exists $c_{ij},c_{ij}'\in\R$ such that
$$T_{\rho}(1_{\Omega_{i}}-1_{\Omega_{j}})(x)=c_{ij},\qquad\forall\,x\in\Sigma_{ij}'.$$
$$T_{\rho}(1_{\Omega_{i}'}-1_{\Omega_{j}'})(x)=c_{ij}',\qquad\forall\,x\in\Sigma_{ij}.$$
\end{lemma}

We denote $N_{ij}(x)$ as the unit exterior normal vector to $\Sigma_{ij}$ for all $1\leq i<j\leq m$.  Also denote $N_{ij}'(x)$ as the unit exterior normal vector to $\Sigma_{ij}'$ for all $1\leq i<j\leq m$.  Let $\Omega_{1},\ldots,\Omega_{m},\Omega_{1}',\ldots,\Omega_{m}'\subset\R^{\adimn}$ be a partition of $\R^{\adimn}$ into measurable sets such that $\partial\Omega_{i},\partial\Omega_{i}'$ are a locally finite union of $C^{\infty}$ manifolds for all $1\leq i\leq m$.  Let $X,X'\in C_{0}^{\infty}(\R^{\adimn},\R^{\adimn})$.  Let $\{\Omega_{i}^{(s)}\}_{s\in(-1,1)}$ be the variation of $\Omega_{i}$ corresponding to $X$ for all $1\leq i\leq m$.  Let $\{\Omega_{i}^{'(s)}\}_{s\in(-1,1)}$ be the variation of $\Omega_{i}'$ corresponding to $X'$ for all $1\leq i\leq m$.  Denote $f_{ij}(x)\colonequals\langle X(x),N_{ij}(x)\rangle$ for all $x\in\Sigma_{ij}$ and $f_{ij}'(x)\colonequals\langle X'(x),N_{ij}'(x)\rangle$ for all $x\in\Sigma_{ij}'$.  We let $N$ denote the exterior pointing unit normal vector to $\redb\Omega_{i}$ for any $1\leq i\leq m$ and we let $N'$ denote the exterior pointing unit normal vector to $\redb\Omega_{i}'$ for any $1\leq i\leq m$.

\begin{lemma}[\embolden{Volume Preserving Second Variation of Minimizers, Multiple Sets}]\label{lemma7rn}
Let $\Omega_{1},\ldots,\Omega_{m},\Omega_{1}',\ldots,\Omega_{m}'\subset\R^{\adimn}$ be two partitions of $\R^{\adimn}$ into measurable sets such that $\partial\Omega_{i},\partial\Omega_{i}'$ are a locally finite union of $C^{\infty}$ manifolds for all $1\leq i\leq m$.  Then
\begin{equation}\label{four32pv2n}
\begin{aligned}
&\frac{\d^{2}}{\d s^{2}}\Big|_{s=0}\sum_{i=1}^{m}\int_{\R^{\adimn}} \int_{\R^{\adimn}} 1_{\Omega_{i}^{(s)}}(y)G(x,y) 1_{\Omega_{i}^{'(s)}}(x)\,\d x\d y\\
&\qquad\qquad\qquad=\sum_{1\leq i<j\leq m}\int_{\Sigma_{ij}'}\Big[\Big(\int_{\redb\Omega_{i}}-\int_{\redb\Omega_{j}}\Big)G(x,y)\langle X(y),N(y)\rangle \,\d y\Big] f_{ij}'(x) \,\d x\\
&\qquad\qquad\qquad\qquad+\sum_{1\leq i<j\leq m}\int_{\Sigma_{ij}}\Big[\Big(\int_{\redb\Omega_{i}'}-\int_{\redb\Omega_{j}'}\Big)G(x,y)\langle X'(y),N'(y)\rangle \,\d y\Big] f_{ij}(x) \,\d x\\
&\qquad\qquad\qquad\qquad\qquad+\int_{\Sigma_{ij}'}\vnormf{\overline{\nabla} T_{\rho}(1_{\Omega_{i}}-1_{\Omega_{j}})(x)}(f_{ij}'(x))^{2} \gamma_{\adimn}(x)\,\d x\\
&\qquad\qquad\qquad\qquad\qquad+\int_{\Sigma_{ij}}\vnormf{\overline{\nabla} T_{\rho}(1_{\Omega_{i}'}-1_{\Omega_{j}'})(x)}(f_{ij}(x))^{2} \gamma_{\adimn}(x)\,\d x.
\end{aligned}
\end{equation}
Also,
\begin{equation}\label{nabeq3n}
\begin{aligned}
\overline{\nabla}T_{\rho}(1_{\Omega_{i}}-1_{\Omega_{j}})(x)&=N_{ij}'(x)\vnormf{\overline{\nabla}T_{\rho}(1_{\Omega_{i}}-1_{\Omega_{j}})(x)},\qquad\forall\,x\in\Sigma_{ij}'.\\
\overline{\nabla}T_{\rho}(1_{\Omega_{i}'}-1_{\Omega_{j}'})(x)&=N_{ij}(x)\vnormf{\overline{\nabla}T_{\rho}(1_{\Omega_{i}'}-1_{\Omega_{j}'})(x)},\qquad\forall\,x\in\Sigma_{ij}.\\
\end{aligned}
\end{equation}
Moreover, $\vnormf{\overline{\nabla} T_{\rho}(1_{\Omega_{i}}-1_{\Omega_{j}})(x)}>0$ for all $x\in\Sigma_{ij}'$, except on a set of Hausdorff dimension at most $\sdimn-1$, and $\vnormf{\overline{\nabla} T_{\rho}(1_{\Omega_{i}'}-1_{\Omega_{j}'})(x)}>0$ for all $x\in\Sigma_{ij}$, except on a set of Hausdorff dimension at most $\sdimn-1$.
\end{lemma}
Equation \eqref{nabeq3n} and the last assertion require a slightly different argument than previously used.  To see the last assertion, note that if there exists $1\leq i<j\leq m$ such that $\vnorm{\overline{\nabla} T_{\rho}(1_{\Omega_{i}}-1_{\Omega_{j}})(x)}=0$ on an open set in $\Sigma_{ij}'$, then choose $X'$ supported in this open set so that the third term of \eqref{four32pv2n} is zero.  Then, choose $Y$ such that sum of the first two terms in \eqref{four32pv2n} is negative.  Multiplying then $X$ by a small positive constant, and noting that the fourth term in \eqref{four32pv2n} has quadratic dependence on $X$, we can create a negative second derivative of the noise stability, giving a contradiction.  We can similarly justify the positive signs appearing in \eqref{nabeq3n} (as opposed to the negative signs from \eqref{nabeq3}).

% if the gradient vanishes, let $X$ supported there
%

Let $v\in\R^{\adimn}$.  For simplicity of notation, we denote $\langle v,N\rangle$ as the collection of functions $(\langle v,N_{ij}\rangle)_{1\leq i<j\leq m}$ and we denote $\langle v,N'\rangle$ as the collection of functions $(\langle v,N_{ij}'\rangle)_{1\leq i<j\leq m}$.  For any $1\leq i<j\leq m$, define
\begin{equation}\label{sdef2n}
\begin{aligned}
S_{ij}(\langle v,N\rangle)(x)
&\colonequals (1-\rho^{2})^{-(\adimn)/2}(2\pi)^{-\frac{\adimn}{2}}\Big(\int_{\partial\Omega_{i}}-\int_{\partial\Omega_{j}}\Big)\langle v,N(y)\rangle e^{-\frac{\vnorm{y-\rho x}^{2}}{2(1-\rho^{2})}}\,\d y,
\,\forall\,x\in\Sigma_{ij'}.\\
S_{ij}'(\langle v,N'\rangle)(x)
&\colonequals (1-\rho^{2})^{-(\adimn)/2}(2\pi)^{-\frac{\adimn}{2}}\Big(\int_{\partial\Omega_{i}'}-\int_{\partial\Omega_{j}'}\Big)\langle v,N'(y)\rangle e^{-\frac{\vnorm{y-\rho x}^{2}}{2(1-\rho^{2})}}\,\d y,
\,\forall\,x\in\Sigma_{ij}.
\end{aligned}
\end{equation}

\begin{lemma}[\embolden{Key Lemma, $m\geq 2$, Translations as Almost Eigenfunctions}]\label{treig2n}
Let $\Omega_{1},\ldots,\Omega_{m},\Omega_{1}',\ldots,\Omega_{m}'\subset\R^{\adimn}$ minimize problem \ref{prob2n}.  Fix $1\leq i<j\leq m$.  Let $v\in\R^{\adimn}$.  Then
\begin{equation}\label{gren}
\begin{aligned}
S_{ij}(\langle v,N\rangle)(x)
&=-\langle v,N_{ij}'(x)\rangle\frac{1}{\rho}\vnormf{\overline{\nabla} T_{\rho}(1_{\Omega_{i}}-1_{\Omega_{j}})(x)},\qquad\forall\,x\in\Sigma_{ij}'.\\
S_{ij}'(\langle v,N'\rangle)(x)
&=-\langle v,N_{ij}(x)\rangle\frac{1}{\rho}\vnormf{\overline{\nabla} T_{\rho}(1_{\Omega_{i}'}-1_{\Omega_{j}'})(x)},\qquad\forall\,x\in\Sigma_{ij}.\\
\end{aligned}
\end{equation}
\end{lemma}
When compared to Lemma \ref{treig2}, Lemma \ref{treig2n} has a negative sign on the right side of the equality, resulting from the positive sign in \eqref{nabeq3n} (as opposed to the negative sign on the right side of \eqref{nabeq3}).  Lemmas \ref{lemma7rn} and \ref{treig2n} then imply the following.

\begin{lemma}[\embolden{Second Variation of Translations, Multiple Sets}]\label{keylemn}
Let $0<\rho<1$.  Let $v\in\R^{\adimn}$.  Let $\Omega_{1},\ldots,\Omega_{m}$ minimize problem \ref{prob2}.  For each $1\leq i\leq m$, let $\{\Omega_{i}^{(s)}\}_{s\in(-1,1)}$ be the variation of $\Omega_{i}$ corresponding to the constant vector field $X\colonequals v$.  Assume that
$$\int_{\partial\Omega_{i}}\langle v,N(x)\rangle \gamma_{\adimn}(x)\,\d x=\int_{\partial\Omega_{i}'}\langle v,N'(x)\rangle \gamma_{\adimn}(x)\,\d x=0,\qquad\forall\,1\leq i\leq m.$$
Then
\begin{flalign*}
&\frac{\d^{2}}{\d s^{2}}\Big|_{s=0}\sum_{i=1}^{m}\int_{\R^{\adimn}}1_{\Omega_{i}^{(s)}}(x)T_{\rho}1_{\Omega_{i}^{'(s)}}(x)\gamma_{\adimn}(x)\,\d x\\
&\qquad\qquad\qquad=\Big(-\frac{1}{\rho}+1\Big)\sum_{1\leq i<j\leq m}\int_{\Sigma_{ij}}\vnormf{\overline{\nabla}T_{\rho}(1_{\Omega_{i}'}-1_{\Omega_{j}'})(x)}\langle v,N_{ij}(x)\rangle^{2}\gamma_{\adimn}(x)\,\d x\\
&\qquad\qquad\qquad\,\,+\Big(-\frac{1}{\rho}+1\Big)\sum_{1\leq i<j\leq m}\int_{\Sigma_{ij}'}\vnormf{\overline{\nabla}T_{\rho}(1_{\Omega_{i}}-1_{\Omega_{j}})(x)}\langle v,N_{ij}'(x)\rangle^{2}\gamma_{\adimn}(x)\,\d x.
\end{flalign*}
\end{lemma}

Since $\rho\in(0,1)$, $-\frac{1}{\rho}+1<0$.  (The analogous inequality in Lemma \ref{keylem} was $\frac{1}{\rho}-1>0$.)  The following Theorem is a modification of the corresponding \cite[Theorem 7.9]{heilman20d}.  In \cite{heilman20d}, the main dimension reduction result in the case of negative correlation obtained a sub-optimal dimension by restricting both $\gamma_{\adimn}(\Omega_{i})$ and $\gamma_{\adimn}(\Omega_{i}')$ for all $1\leq i\leq m$.  Here we obtain the optimal dimension dependence (concluding that the $\Theta$ sets below are contained in $\R^{m-1}$, which is optimal), by instead imposing the restriction that $\gamma_{\adimn}(\Omega_{i})=\gamma_{\adimn}(\Omega_{i}')$ for all $1\leq i\leq m$.  In fact, the argument below seems to work with arbitrary linear constraints on the measures of the sets $\Omega_{i}$ and $\Omega_{i}'$.

\begin{theorem}[\embolden{Main Structure Theorem/ Dimension Reduction, Negative Correlation}]\label{mainthm1n}
Fix $0<\rho<1$.  Let $m\geq2$ with $m\leq \sdimn+2$.  Let $\Omega_{1},\ldots\Omega_{m},\Omega_{1}',\ldots\Omega_{m}'\subset\R^{\adimn}$ minimize Problem \ref{prob2n} (that exist by Lemma \ref{existlemn}).  Then, after rotating the sets $\Omega_{1},\ldots\Omega_{m},$ $\Omega_{1}',\ldots\Omega_{m}'$ and applying Lebesgue measure zero changes to these sets, there exist measurable sets $\Theta_{1},\ldots\Theta_{m},\Theta_{1}',\ldots\Theta_{m}'\subset\R^{m-1}$ such that,
$$\Omega_{i}=\Theta_{i}\times\R^{\sdimn-m+2},\,\,\Omega_{i}'=\Theta_{i}'\times\R^{\sdimn-m+2}\qquad\forall\, 1\leq i\leq m.$$  %  m-1 + n-m+1=n
\end{theorem}
\begin{proof}
The sets $\Omega_{1},\ldots\Omega_{m},\Omega_{1}',\ldots\Omega_{m}'\subset\R^{\adimn}$ with $\gamma_{\adimn}(\Omega_{i})=\gamma_{\adimn}(\Omega_{i}')$ for all $1\leq i\leq m$ that minimize Problem \ref{prob2} exist by Lemma \ref{existlemn}.  From Lemma \ref{reglemn} their boundaries are locally finite unions of $C^{\infty}$ $\sdimn$-dimensional manifolds.

By Lemma \ref{firstvarmaxnsn}, for all $1\leq i<j\leq m$, there exists $c_{ij},c_{ij}'\in\R$ such that
$$T_{\rho}(1_{\Omega_{i}}-1_{\Omega_{j}})(x)=c_{ij},\qquad\forall\,x\in\Sigma_{ij}'.$$
$$T_{\rho}(1_{\Omega_{i}'}-1_{\Omega_{j}'})(x)=c_{ij}',\qquad\forall\,x\in\Sigma_{ij}.$$
By this condition, the regularity Lemma \ref{reglemn}, and the last part of Lemma \ref{lemma7rn},

\begin{flalign*}
\overline{\nabla}T_{\rho}(1_{\Omega_{i}}-1_{\Omega_{j}})(x)&=N_{ij}'(x)\vnormf{\overline{\nabla}T_{\rho}(1_{\Omega_{i}}-1_{\Omega_{j}})(x)},\qquad\forall\,x\in\Sigma_{ij}'.\\
\overline{\nabla}T_{\rho}(1_{\Omega_{i}'}-1_{\Omega_{j}'})(x)&=N_{ij}(x)\vnormf{\overline{\nabla}T_{\rho}(1_{\Omega_{i}'}-1_{\Omega_{j}'})(x)},\qquad\forall\,x\in\Sigma_{ij}.\\
\end{flalign*}
Moreover, by the last part of Lemma \ref{lemma7rn}, except for sets $\sigma_{ij},\sigma_{ij}'$ of Hausdorff dimension at most $\sdimn-1$, we have
\begin{equation}\label{nine1}
\begin{aligned}
\vnormf{\overline{\nabla}T_{\rho}(1_{\Omega_{i}}-1_{\Omega_{j}})(x)}&>0,\qquad\forall\,x\in\Sigma_{ij}'\setminus\sigma_{ij}'.\\
\vnormf{\overline{\nabla}T_{\rho}(1_{\Omega_{i}'}-1_{\Omega_{j}'})(x)}&>0,\qquad\forall\,x\in\Sigma_{ij}\setminus\sigma_{ij}.\\
\end{aligned}
\end{equation}

Fix $v\in\R^{\adimn}$, and consider the variation of $\Omega_{1},\ldots,\Omega_{m},\Omega_{1}',\ldots,\Omega_{m}'$ induced by the constant vector field $X\colonequals v$.  For all $1\leq i<j\leq m$, define $S_{ij}$ as in \eqref{sdef2n}.  Define
\begin{flalign*}
V&\colonequals\Big\{v\in\R^{\adimn}\colon \sum_{j\in\{1,\ldots,m\}\setminus\{i\}}\int_{\Sigma_{ij}}\langle v,N_{ij}(x)\rangle \gamma_{\adimn}(x)\,\d x\\
&\qquad\qquad\qquad\qquad
=\sum_{j\in\{1,\ldots,m\}\setminus\{i\}}\int_{\Sigma_{ij}'}\langle v,N_{ij}'(x)\rangle \gamma_{\adimn}(x)\,\d x,\qquad\forall\,1\leq i\leq m\Big\}.
\end{flalign*}
From Lemma \ref{keylemn},
\begin{flalign*}
v\in V\,\Longrightarrow&\,\,\,\frac{\d^{2}}{\d s^{2}}\Big|_{s=0}\sum_{i=1}^{m}\int_{\R^{\adimn}}1_{\Omega_{i}^{(s)}}(x)T_{\rho}1_{\Omega_{i}^{'(s)}}(x)\gamma_{\adimn}(x)\,\d x\\
&\qquad\qquad=\Big(-\frac{1}{\rho}+1\Big)\sum_{1\leq i<j\leq m}\int_{\Sigma_{ij}}\vnormf{\overline{\nabla}T_{\rho}(1_{\Omega_{i}'}-1_{\Omega_{j}'})(x)}\langle v,N_{ij}(x)\rangle^{2}\gamma_{\adimn}(x)\,\d x\\
&\qquad\qquad\,\,+\Big(-\frac{1}{\rho}+1\Big)\sum_{1\leq i<j\leq m}\int_{\Sigma_{ij}'}\vnormf{\overline{\nabla}T_{\rho}(1_{\Omega_{i}}-1_{\Omega_{j}})(x)}\langle v,N_{ij}'(x)\rangle^{2}\gamma_{\adimn}(x)\,\d x.
\end{flalign*}
Since $0<\rho<1$, \eqref{nine1} implies
\begin{equation}\label{nine2}
v\in V\,\Longrightarrow\,\langle v,N_{ij}(x)\rangle=\langle v,N_{ij}'(x)\rangle=0,\qquad\forall\,x\in\Sigma_{ij},\,\forall\,x'\in\Sigma_{ij}',\,\forall\,1\leq i<j\leq m.
\end{equation}
The set $V$ has dimension at least $\sdimn+2-m$, by the rank-nullity theorem, since $V$ is the null space of the linear operator $M\colon \R^{\adimn}\to\R^{m}$ defined by
$$
(M(v))_{i}\colonequals \sum_{j\in\{1,\ldots,m\}\setminus\{i\}}\Big(\int_{\Sigma_{ij}}\langle v,N_{ij}(x)\rangle \gamma_{\adimn}(x)\,\d x
-\int_{\Sigma_{ij}}\langle v,N_{ij}'(x)\rangle \gamma_{\adimn}(x)\,\d x\Big),\,\forall\,1\leq i\leq m
$$ %  n+1-(m-1)=n-m+2
and $M$ has rank at most $m-1$ (since $\sum_{i=1}^{m}(M(v))_{i}=0$ for all $v\in\R^{\adimn}$).  So, by \eqref{nine2}, after rotating $\Omega_{1},\ldots,\Omega_{m},\Omega_{1}',\ldots,\Omega_{m}'$, we conclude $\exists$ measurable $\Theta_{1},\ldots,\Theta_{m},\Theta_{1}',\ldots,\Theta_{m}'\subset\R^{m-1}$ such that
$$\Omega_{i}=\Theta_{i}\times\R^{\sdimn+2-m},\quad \Omega_{i}'=\Theta_{i}'\times\R^{\sdimn+2-m}\qquad\forall\,1\leq i\leq m.$$
\end{proof}

\section{Bilinear Version: Dilation}

\begin{lemma}[\embolden{Dilation as Almost Eigenfunction}]\label{mceign}
Let $\Omega_{1},\ldots,\Omega_{m},\Omega_{1}',\ldots,\Omega_{m}'\subset\R^{\adimn}$ minimize Problem \ref{prob2n}.  Then for all $1\leq i,j\leq m$,
\begin{equation}\label{six1zn}
\begin{aligned}
&S_{ij}(\langle\cdot ,N\rangle)(x)+\langle x,N_{ij}'(x)\rangle\vnormf{\overline{\nabla}T_{\rho}(1_{\Omega_{i}}-1_{\Omega_{j}})(x)}\\
&\qquad=\Big(\frac{1}{\rho^{2}}-1\Big)\Big(-\langle x,N_{ij}'(x)\rangle\vnorm{\overline{\nabla} T_{\rho}(1_{\Omega_{i}}-1_{\Omega_{j}})(x)}
+\rho \frac{\d}{\d\rho}T_{\rho}(1_{\Omega_{i}}-1_{\Omega_{j}})(x)\Big),\quad\forall\,x\in\Sigma_{ij}'.
 \end{aligned}
\end{equation}
\begin{equation}\label{six1zn2}
\begin{aligned}
&S_{ij}'(\langle\cdot ,N'\rangle)(x)+\langle x,N_{ij}(x)\rangle\vnormf{\overline{\nabla}T_{\rho}(1_{\Omega_{i}'}-1_{\Omega_{j}'})(x)}\\
&\qquad=\Big(\frac{1}{\rho^{2}}-1\Big)\Big(-\langle x,N_{ij}(x)\rangle\vnormf{\overline{\nabla} T_{\rho}(1_{\Omega_{i}'}-1_{\Omega_{j}'})(x)}
+\rho \frac{\d}{\d\rho}T_{\rho}(1_{\Omega_{i}'}-1_{\Omega_{j}'})(x)\Big),\quad\forall\,x\in\Sigma_{ij}.
 \end{aligned}
\end{equation}
\end{lemma}
\begin{proof}
We prove \eqref{six1zn}.  From \eqref{nabeq3n}, for all $1\leq i,j\leq m$,
\begin{equation}\label{firstven}
\begin{aligned}
\overline{\nabla}T_{\rho}(1_{\Omega_{i}}-1_{\Omega_{j}})(x)&=N_{ij}'(x)\vnormf{\overline{\nabla}T_{\rho}(1_{\Omega_{i}}-1_{\Omega_{j}})(x)},\qquad\forall\,x\in\Sigma_{ij}'.\\
\overline{\nabla}T_{\rho}(1_{\Omega_{i}'}-1_{\Omega_{j}'})(x)&=N_{ij}(x)\vnormf{\overline{\nabla}T_{\rho}(1_{\Omega_{i}'}-1_{\Omega_{j}'})(x)},\qquad\forall\,x\in\Sigma_{ij}.\\
\end{aligned}
\end{equation}
Taking the divergence of \eqref{firstven}, for all $1\leq i,j\leq m$,
\begin{equation}\label{firstvev2p}
\begin{aligned}
\mathrm{div}\overline{\nabla} T_{\rho}(1_{\Omega_{i}}-1_{\Omega_{j}})(x)
&=\mathrm{div}(N_{ij}'(x))\vnorm{\overline{\nabla}T_{\rho}(1_{\Omega_{i}}-1_{\Omega_{j}})(x)}\\
&\qquad\qquad\qquad\langle N_{ij}'(x),\overline{\nabla}\vnorm{\overline{\nabla} T_{\rho}(1_{\Omega_{i}}-1_{\Omega_{j}})(x)} \rangle.
\end{aligned}
\end{equation}
Applying the divergence theorem to the last equality in \eqref{six3},
\begin{equation}\label{six3tp}
\begin{aligned}
&\mathrm{div}\overline{\nabla} T_{\rho}(1_{\Omega_{i}}-1_{\Omega_{j}})(x)\\
&\qquad=(1-\rho^{2})^{-(\adimn)/2}(2\pi)^{-(\adimn)/2}\frac{-\rho^{2}}{1-\rho^{2}}\Big(\int_{\partial\Omega_{i}}-\int_{\partial\Omega_{j}}\Big)\Big\langle (y-\rho x), N(y)\Big\rangle e^{-\frac{\vnorm{y-\rho x}^{2}}{2(1-\rho^{2})}} \,\d y\\
&\qquad\stackrel{\eqref{sdef2n}}{=}-\frac{\rho^{2}}{1-\rho^{2}}\Big(S_{ij}(\langle \cdot, N\rangle)(x)  -\rho \langle x,S_{ij}(N)(x)\rangle\Big)\\
&\qquad\stackrel{\eqref{gren}}{=}-\frac{\rho^{2}}{1-\rho^{2}}\Big(S_{ij}(\langle \cdot, N\rangle)(x)  +\langle x,N_{ij}'(x)\rangle\vnorm{\overline{\nabla} T_{\rho}(1_{\Omega_{i}}-1_{\Omega_{j}})}\Big).
\end{aligned}
\end{equation}
This equation and \eqref{firstvev2p} proves \eqref{six1zn}, together with
\begin{flalign*}
&\mathrm{div}\overline{\nabla} T_{\rho}(1_{\Omega_{i}}-1_{\Omega_{j}})(x)\\
&\qquad\qquad=\overline{\Delta}T_{\rho}(1_{\Omega_{i}}-1_{\Omega_{j}})(x)
-\langle x,\nabla T_{\rho}(1_{\Omega_{i}}-1_{\Omega_{j}})(x)\rangle
+\langle x,\overline{\nabla} T_{\rho}(1_{\Omega_{i}}-1_{\Omega_{j}})(x)\rangle\\
&\qquad\qquad\stackrel{\eqref{oup}\wedge\eqref{firstven}}{=}-\rho\frac{d}{d\rho}T_{\rho}(1_{\Omega_{i}}-1_{\Omega_{j}})(x)+\langle x,N_{ij}'(x)\rangle\vnormf{\overline{\nabla} T_{\rho}(1_{\Omega_{i}}-1_{\Omega_{j}})(x)}.
\end{flalign*}
Equation \eqref{six1zn2} is proven analogously.  A priori finiteness follows from Remark \ref{drk}.
\end{proof}

%\snote{still need to justify use of divergence theorem here and above}

Using the dilation as an eigenfunction for hyperstable sets implies that those sets are dilation invariant.

\begin{lemma}[\embolden{Hyperstable Implies Dilation Invariance}]\label{lemma50}
Let $\Omega_{1},\ldots,\Omega_{m},\Omega_{1}',\ldots,\Omega_{m}'\subset\R^{\adimn}$ minimize Problem \ref{prob2n}.  Assume $\exists$ measurable $\Theta_{1},\ldots,\Theta_{m},$ \\$\Theta_{1}',\ldots,\Theta_{m}'\subset\R^{m-1}$ such that
$$\Omega_{i}=\Theta_{i}\times\R,\quad \Omega_{i}'=\Theta_{i}'\times\R\qquad\forall\,1\leq i\leq m.$$
Assume also that
\begin{flalign*}
&\sum_{1\leq i<j\leq m}\Big[\int_{\Sigma_{ij}'}\langle x,N_{ij}'(x)\rangle\frac{\d}{\d\rho}T_{\rho}(1_{\Omega_{i}}-1_{\Omega_{j}})(x)\gamma_{\adimn}(x) \,\d x\\
&\qquad\qquad\qquad\qquad\qquad\qquad\qquad+\int_{\Sigma_{ij}}\langle x,N_{ij}(x)\rangle\frac{\d}{\d\rho}T_{\rho}(1_{\Omega_{i}'}-1_{\Omega_{j}'})(x)\gamma_{\adimn}(x) \,\d x\Big]\leq0.
\end{flalign*}
Then $\Omega_{1},\ldots,\Omega_{m},\Omega_{1}',\ldots,\Omega_{m}'$ are dilation invariant.  That is,
$$t\Omega_{i}=\Omega_{i},\qquad t\Omega_{i}'=\Omega_{i}',\qquad \forall\,t>0\,\,\forall\,1\leq i\leq m.$$
\end{lemma}
\begin{proof}
Consider the vector field $X\colon\R^{\adimn}\to\R$ defined by
$$X(x)\colonequals x_{\adimn}\cdot x,\qquad\forall\,x\in\R^{\adimn}.$$
Note that $\langle X,N_{ij}(x)\rangle$ and $\langle X,N_{ij}'(x)\rangle$ is constant as $x_{\adimn}$ varies, and $x_{\adimn}$ is constant as $x_{1},\ldots,x_{\sdimn}$ varies.  We can therefore apply Fubini's Theorem when we apply $S_{ij}$ and $S_{ij}'$ to the variation corresponding to $X$, and also note that $\int_{\redb\Omega_{i}}\langle X(x),N(x)\rangle\gamma_{\adimn}(x)\,\d x=0$ and $\int_{\redb\Omega_{i}'}\langle X(x),N'(x)\rangle\gamma_{\adimn}(x)\,\d x=0$ for all $1\leq i\leq m$ since $\int_{\R}x_{\adimn}\gamma_{1}(x_{\adimn})\,\d x_{\adimn}=0$.  From \eqref{sdef2n}, and using the well-known property of Hermite polynomials on the real line that $T_{\rho}x_{1}=\rho x_{1}$ for all $x_{1}\in\R$, we have for all $1\leq i,j\leq m$,
\begin{flalign*}
S_{ij}(\langle X,N\rangle)(x)&=\rho x_{\adimn}S_{ij}(\langle \cdot, N(\cdot)\rangle)(x),\qquad\forall\,x\in\Sigma_{ij}'.\\
S_{ij}'(\langle X,N'\rangle)(x)&=\rho x_{\adimn}S_{ij}(\langle \cdot, N'(\cdot)\rangle)(x),\qquad\forall\,x\in\Sigma_{ij}.
\end{flalign*}
Combining this with Lemma \ref{mceign}, for all $1\leq i,j\leq m$,
\begin{equation}\label{seven1n}
\begin{aligned}
S_{ij}(\langle X,N\rangle)(x)
&=-\frac{1}{\rho}\langle X,N'\rangle\vnormf{\overline{\nabla}T_{\rho}(1_{\Omega_{i}}-1_{\Omega_{j}})(x)}
+(1-\rho^{2})\frac{\d}{\d\rho}T_{\rho}(1_{\Omega_{i}}-1_{\Omega_{j}})(x),\,\forall\,x\in\Sigma_{ij}'.\\
S_{ij}'(\langle X,N'\rangle)(x)
&=-\frac{1}{\rho}\langle X,N\rangle\vnormf{\overline{\nabla}T_{\rho}(1_{\Omega_{i}'}-1_{\Omega_{j}'})(x)}
+(1-\rho^{2})\frac{\d}{\d\rho}T_{\rho}(1_{\Omega_{i}'}-1_{\Omega_{j}'})(x),\,\forall\,x\in\Sigma_{ij}.
\end{aligned}
\end{equation}%
Combining these facts and Lemma \ref{lemma7rn}, and using also $\int_{\R}x_{\adimn}^{2}\gamma_{1}(x_{\adimn})\,\d x_{\adimn}=1$ and that $\vnormf{\overline{\nabla}T_{\rho}1_{\Omega}(x)}$ is constant as $x_{\adimn}$ varies, the variation of $\{\Omega_{i}^{(s)}\}_{s\in(-1,1)},\{\Omega_{i}^{'(s)}\}_{s\in(-1,1)}$ corresponding to this choice of $X$ satisfies
\begin{equation}\label{seven2n}
\begin{aligned}
&\frac{\d^{2}}{\d s^{2}}\Big|_{s=0}\sum_{i=1}^{m}\int_{\R^{\adimn}} \int_{\R^{\adimn}} 1_{\Omega_{i}^{(s)}}(y)G(x,y) 1_{\Omega_{i}^{'(s)}}(x)\,\d x\d y\\
&\qquad=\sum_{1\leq i<j\leq m}\int_{\Sigma_{ij}'} \Big(S_{ij}(\langle X,N\rangle)(x)f_{ij}'(x)+\vnormf{\overline{\nabla} T_{\rho}(1_{\Omega_{i}}-1_{\Omega_{j}})(x)}(f_{ij}'(x))^{2} \Big) \,\d x\\
&\qquad\qquad+\sum_{1\leq i<j\leq m}\int_{\Sigma_{ij}}\Big(S_{ij}'(\langle X,N'\rangle)(x) f_{ij}(x)+\vnormf{\overline{\nabla} T_{\rho}(1_{\Omega_{i}'}-1_{\Omega_{j}'})(x)}(f_{ij}(x))^{2}\Big) \,\d x\\
&\qquad\stackrel{\eqref{seven1n}}{=}
\Big(-\frac{1}{\rho}+1\Big)\sum_{1\leq i<j\leq m}\Big[\int_{\Sigma_{ij}'} \langle x,N_{ij}'(x)\rangle^{2}\vnormf{\overline{\nabla} T_{\rho}(1_{\Omega_{i}}-1_{\Omega_{j}})(x)}\gamma_{\adimn}(x)\,\d x\\
&\qquad\qquad\qquad\qquad\qquad\qquad\qquad
+\int_{\Sigma_{ij}} \langle x,N_{ij}(x)\rangle^{2}\vnormf{\overline{\nabla} T_{\rho}(1_{\Omega_{i}'}-1_{\Omega_{j}'})(x)}\gamma_{\adimn}(x)\,\d x\Big]\\
&\qquad+(1-\rho^{2})\sum_{1\leq i<j\leq m}\Big[\int_{\Sigma_{ij}'}\langle x,N_{ij}'(x)\rangle\frac{\d}{\d\rho}T_{\rho}(1_{\Omega_{i}}-1_{\Omega_{j}})(x)\gamma_{\adimn}(x) \,\d x\\
&\qquad\qquad\qquad\qquad\qquad\qquad\qquad+\int_{\Sigma_{ij}}\langle x,N_{ij}(x)\rangle\frac{\d}{\d\rho}T_{\rho}(1_{\Omega_{i}'}-1_{\Omega_{j}'})(x)\gamma_{\adimn}(x) \,\d x\Big].
\end{aligned}
\end{equation}
By assumption, the last quantity is nonpositive.  By the minimization property of the sets $\Omega_{1},\ldots,\Omega_{m},\Omega_{1}',\ldots,\Omega_{m}'$, we have
$$\frac{\d^{2}}{\d s^{2}}\Big|_{s=0}\sum_{i=1}^{m}\int_{\R^{\adimn}} \int_{\R^{\adimn}} 1_{\Omega_{i}^{(s)}}(y)G(x,y) 1_{\Omega_{i}^{'(s)}}(x)\,\d x\d y\geq0.$$
In order for this inequality to be true, since $0<\rho<1$, $1-(1/\rho)<0$, we must have
$$\langle x,N_{ij}'(x)\rangle=0,\qquad\forall\,x\in\Sigma_{ij}',\quad\forall\,1\leq i<j\leq m.$$
$$\langle x,N_{ij}(x)\rangle=0,\qquad\forall\,x\in\Sigma_{ij},\quad\forall\,1\leq i<j\leq m.$$
That is, the sets $\Omega_{1},\ldots,\Omega_{m},\Omega_{1}',\ldots,\Omega_{m}'$ are dilation invariant.
\end{proof}

\begin{lemma}[\embolden{Dilation Invariance Implies Simplicial}]\label{lemma51}
Let $m\leq 4$.  Assume that $\Omega_{1},\ldots,\Omega_{m},\Omega_{1}',\ldots,\Omega_{m}'\subset\R^{m}$ are dilation invariant partitions and hyperstable, so that
$$t\Omega_{i}=\Omega_{i},\qquad t\Omega_{i}'=\Omega_{i}',\qquad \forall\,t>0\,\,\forall\,1\leq i\leq m.$$
Then $\Omega_{1},\ldots,\Omega_{m},\Omega_{1}',\ldots,\Omega_{m}'$ are congruent, regular simplicial cones.
\end{lemma}
\begin{proof}
\textbf{Case of Dimension One}.  If it occurs that $\Omega_{1},\ldots,\Omega_{m},\Omega_{1}',\ldots,\Omega_{m}'\subset\R$, then dilation invariance implies that two of the sets $\Omega_{1},\ldots,\Omega_{m}$ are opposing half lines and the rest of these sets are empty.  Similarly, two of the sets $\Omega_{1}',\ldots,\Omega_{m}'$ are opposing half lines and the rest are empty.  Up to relabeling the sets, there are only two possibly configurations of the partitions, namely
$$\Omega_{1}=(-\infty,0],\quad \Omega_{2}=(0,\infty),\qquad\Omega_{1}'=(-\infty,0],\quad \Omega_{2}'=(0,\infty),\quad\mbox{or}$$
$$\Omega_{1}=(-\infty,0],\quad \Omega_{2}=(0,\infty),\qquad\Omega_{1}'=(0,\infty),\quad \Omega_{2}'=(-\infty,0].\qquad$$
The latter case has smaller noise stability than the former case, so the latter case completes the proof when the sets are all contained in $\R$.

\textbf{Case of $m=2$ sets}.  By Theorem \ref{mainthm1n}, we may assume that $\Omega_{1},\ldots,\Omega_{m},\Omega_{1}',\ldots,\Omega_{m}'\subset\R$, so the proof is completed by the previous case.

\textbf{Case of $m=3$ Sets}.  By Theorem \ref{mainthm1n}, we may assume that $\Omega_{1},\ldots,\Omega_{m},\Omega_{1}',\ldots,\Omega_{m}'\subset\R^{2}$.  Dilation Invariance of these sets implies that each set is a finite union of sectors, and for any $1\leq i<j\leq m$, $(\partial\Omega_{i})\cap(\partial\Omega_{j})$ is a union of half-lines emanating from the origin (and similarly for $(\partial\Omega_{i}')\cap(\partial\Omega_{j}')$.)  For some fixed $1\leq i<j\leq m$, let $L'$ be a half line in $(\partial\Omega_{i})\cap(\partial\Omega_{j})$ and consider the function $T_{\rho}(1_{\Omega_{i}'}-1_{\Omega_{j}'})(x)$ where $x$ varies in the line $L$ containing $L'$.  The First Variation condition in Lemma \ref{firstvarmaxnsn} implies that $T_{\rho}(1_{\Omega_{i}'}-1_{\Omega_{j}'})(x)$ is constant on the half line $L'$.  In particular, we have by dilation invariance and \eqref{oudef},
\begin{equation}\label{five1}
\gamma_{2}(\Omega_{i}')-\gamma_{2}(\Omega_{j}')=T_{\rho}(1_{\Omega_{i}'}-1_{\Omega_{j}'})(0).
\end{equation}
Also, as $x\in L'$ tends to infinity, since $\Omega_{i}'$ and $\Omega_{j}'$ are finite unions of sectors, we have
$$\lim_{x\in L',\, \vnorm{x}\to\infty}T_{\rho}(1_{\Omega_{i}'}-1_{\Omega_{j}'})(x)\in\{0,1\}.$$
If this limit is $1$, then \eqref{five1} is violated.  We conclude that
\begin{equation}\label{five2}
\lim_{x\in L',\, \vnorm{x}\to\infty}T_{\rho}(1_{\Omega_{i}'}-1_{\Omega_{j}'})(x)=0.
\end{equation}
Then \eqref{five1} is also equal to zero since $T_{\rho}(1_{\Omega_{i}'}-1_{\Omega_{j}'})(x)$ is constant for all $x\in L'$.  Consequently, either $\gamma_{2}(\Omega_{i})=\gamma_{2}(\Omega_{i}')=1/3$ for all $1\leq i\leq 3$, or some $\Omega_{i}$ and $\Omega_{j}'$ are empty, in which case we return to the previous $m=2$ case.

So, for the remainder of the proof, we may assume that
$$\gamma_{2}(\Omega_{i})=\gamma_{2}(\Omega_{i}')=1/3,\qquad\forall\,1\leq i\leq 3.$$

Since $\Omega_{i}'$ and $\Omega_{j}'$ are finite unions of sectors, in order to have $T_{\rho}(1_{\Omega_{i}'}-1_{\Omega_{j}'})(x)=0$ for all $x$ in the half line $L'$, it must be the case that $\Omega_{i}'$ is the reflection of $\Omega_{j}'$ across $L'$.  Therefore, with $1\leq i<j\leq m$ fixed, let $A_{ij}'$ denote the group of all such reflections that interchange $\Omega_{i}'$ and $\Omega_{j}'$.  (If $A_{ij}'$ is empty, then $\Omega_{i}\cap\Omega_{j}=\emptyset$.)  If $A_{ij}'$ is nonempty with more than two elements, then the composition of two distinct reflections is a rotation of $\Omega_{i}'$ that fixes $\Omega_{i}'$.  Since $\Omega_{i}'$ is invariant under a non-identity rotation, we conclude that $\int_{\Omega_{i}'}x\gamma_{2}(x)\,\d x=0$.  Lemma \ref{keylemn} then implies that we can  reduce back to the dimension one case.

Having ruled out the cases that $A_{ij}'$ is empty or it has more than two elements, we are left with the case that $A_{ij}'$ has exactly two elements (i.e. the identity and one reflection).  That is, we assume for all $1\leq i<j\leq 3$, $\Omega_{i}'\cap\Omega_{j}'$ is either a single half line, or a pair of parallel half lines.  The latter case implies that $\int_{\Omega_{i}'}x\gamma_{2}(x)\,\d x=0$.  The former case implies that $\Omega_{1}',\ldots,\Omega_{3}'$ are three $120$ degree sectors.

The case $m=3$ is therefore concluded.

\textbf{Case of $m=4$ Sets}.  By Theorem \ref{mainthm1n}, we may assume that $\Omega_{1},\ldots,\Omega_{m},\Omega_{1}',\ldots,\Omega_{m}'\subset\R^{3}$.  As in the $m=3$ case, we may deduce that
$$\gamma_{3}(\Omega_{i})=\gamma_{3}(\Omega_{i}')=1/4,\qquad\forall\,1\leq i\leq 4.$$
As in the $m=3$ case, in order to have $T_{\rho}(1_{\Omega_{i}'}-1_{\Omega_{j}'})(x)=0$ for all $x$ in the boundary between $\Omega_{i}'$ and $\Omega_{j}'$, it must be the case that $\Omega_{i}'$ is the reflection of $\Omega_{j}'$ across a plane containing a part of this boundary.

For any $1\leq i<j\leq m$ fixed, let $A_{ij}'$ denote the group of all such reflections that interchange $\Omega_{i}'$ and $\Omega_{j}'$.  (If $A_{ij}'$ is empty, then $\Omega_{i}\cap\Omega_{j}=\emptyset$.)  If $A_{ij}'$ is nonempty with more than two elements, then the composition of two distinct reflections is a rotation of $\Omega_{i}'$ that fixes $\Omega_{i}'$.  Since $\Omega_{i}'$ is invariant under a non-identity rotation, we conclude that $\int_{\Omega_{i}'}x\gamma_{2}(x)\,\d x=0$.  Lemma \ref{keylemn} then implies that we can reduce to the dimension two case.  The dimension two case then proceeds as in the $m=3$ case.

Having ruled out the cases that $A_{ij}'$ is empty or it has more than two elements, we are left with the case that $A_{ij}'$ has exactly two elements (i.e. the identity and one reflection).  That is, we assume for all $1\leq i<j\leq 4$, $\Omega_{i}'\cap\Omega_{j}'$ is a disc sector.  We then conclude that $\Omega_{1}',\ldots,\Omega_{4}'$ are congruent regular tetrahedral cones.  (Since $T_{\rho}(1_{\Omega_{i}'}-1_{\Omega_{j}'})(x)=0$ for all $x$ in the boundary between $\Omega_{i}'$ and $\Omega_{j}'$, if we particularly choose a line that is the intersection of three of the sets, then it must be the case that the three sets meet at $120$ degree angles.)

%Consider now the group of automorphisms ....
\end{proof}
\begin{remark}\label{m5rk}
When $m\geq5$, unlike in the case $m\leq 4$, there seems to be no simple a priori reason that the interfaces between the sets $\Omega_{1},\ldots,\Omega_{m},\Omega_{1}',\ldots,\Omega_{m}'\subset\R^{\adimn}$ are flat.  We expect this should be true, but we cannot presently prove it.
\end{remark}

Lemmas \ref{lemma50} and \ref{lemma51} together imply the following, which is a restatement of Theorem \ref{main1}.

\begin{theorem}\label{main16}
Let $m\leq 4$.  Suppose $\Omega_{1},\ldots,\Omega_{m},\Omega_{1}',\ldots,\Omega_{m}'\subset\R^{\adimn}$ minimize Problem \ref{prob2n} and these partitions are hyperstable.  Then $\Omega_{1},\ldots,\Omega_{m},\Omega_{1}',\ldots,\Omega_{m}'$ are congruent, regular simplicial cones.
\end{theorem}

%\section{Conditional MAX-3-CUT Hardness}

%$$I(\rho)=I_{m}(\rho)=\int_{\R}\int_{\R}G(x,y) [F(x,y)]^{m-1}\,\d x \d y.$$
%$$F(x,y)\colonequals\int_{-\infty}^{x}\int_{-\infty}^{y}G(a,b)\,\d a \d b.$$
%$$I_{2}(\rho)=\frac{1}{2}(1-\frac{1}{\pi}\cos^{-1}(\rho)).$$
%%
%$$I_{3}(\rho)=\frac{1}{9}+\frac{1}{4\pi^{2}}[(\cos^{-1}(-\rho))^{2}-(\cos^{-1}(\rho/2))^{2}].$$
%% https://www.sciencedirect.com/science/article/pii/S0022000003001454
%%https://dr.ntu.edu.sg/bitstream/10356/95230/1/14.%20On%20approximate%20graph%20colouring%20and%20MAX-k-CUT%20algorithms%20based%20on%20the%20%CE%BD-function.pdf
%
%$$\min_{-\frac{1}{m-1}\leq\rho\leq 1}\frac{m}{m-1}\frac{1-m I(\rho)}{1-\rho}=\min_{-\frac{1}{m-1}\leq\rho\leq 0}\frac{m}{m-1}\frac{1-m I(\rho)}{1-\rho}$$
%When $m=2$, get
%$$
%\min_{-1\leq\rho\leq 0}2\frac{1-m I(\rho)}{1-\rho}
%=\min_{-1\leq\rho\leq 0}\frac{2}{\pi}\frac{\cos^{-1}(\rho)}{1-\rho}
%\approx.878567.
%$$
%%%%%t=linspace(-1,1,1000);
%%t=linspace(-.7,-.68,1000);
%%plot(t, (2/pi)*(acos(t))./(1-t) )
%%%%%axis([-1,1,-1,1])
%
%When $m=3$, get
%$$
%\min_{-\frac{1}{m-1}\leq\rho\leq 0}\frac{m}{m-1}\frac{1-m I(\rho)}{1-\rho}
%=\frac{3}{2}\frac{1-m I(-1/2)}{1-(-1/2)}
%=\frac{2}{3}-\frac{3}{4\pi^{2}}[(\cos^{-1}(1/2))^{2}-(\cos^{-1}(-1/4))^{2}]
%\approx .836008.
%$$
%%t=linspace(-1/2,1,1000);
%%plot(t,  (3/2)*(1-3*((1/9)+  ((acos(-t)).^2-(acos(t/2)).^2)/(4*pi^2)))./(1-t) )
%%axis([-1,1,-1,1])
%%
%%  (2/3)-(3/(4*pi^2))*((acos(1/2))^2 -(acos(-1/4))^2)

\section{One Set, Negative Correlation}

Here we single out the case that $m=2$ for negative correlation, since this case can be solved exactly.  This case follows from Borell's inequality, but a variational proof of this inequality has not been given before.  In the positive correlation case, this proof appeared already in \cite{heilman20d,heilman21}, but some small changes are needed to solve the negative correlation case.  In particular, we only impose the constraint that both sets have the same measure, without further constraining what that measure could be.  Such a constraint is most natural for applications to hardness of the MAX-CUT problem \cite{khot07,isaksson11}.

Problem \ref{prob2n} in the case $m=2$ can be restated in the following way.

\begin{prob}\label{onesetprob}
Let $\Omega,\Omega'\subset\R^{\adimn}$ be measurable sets minimizing
$$\int_{\R^{\adimn}}1_{\Omega}(x)T_{\rho}1_{\Omega'}(x)\gamma_{\adimn}(x)\,\d x
+\int_{\R^{\adimn}}1_{\Omega^{c}}(x)T_{\rho}1_{(\Omega')^{c}}(x)\gamma_{\adimn}(x)\,\d x$$
subject to the constraint that $\gamma_{\adimn}(\Omega)=\gamma_{\adimn}(\Omega')$.
\end{prob}

Theorem \ref{mainthm1n} says that we may assume that $\Omega,\Omega'\subset\R$ in order to solve Problem \ref{onesetprob}.

Since $\Omega_{1}=\Omega$ and $\Omega_{2}=\Omega^{c}$, Equation \ref{sdef2n} can be simplified to
\begin{equation}\label{sdef2noneset}
\begin{aligned}
S(\langle v,N\rangle)(x)
&\colonequals (1-\rho^{2})^{-(\adimn)/2}(2\pi)^{-(\adimn)/2}\int_{\partial\Omega}2\langle v,N(y)\rangle e^{-\frac{\vnorm{y-\rho x}^{2}}{2(1-\rho^{2})}}\,\d y,
\,\forall\,x\in\Sigma'\\
S'(\langle v,N'\rangle)(x)
&\colonequals (1-\rho^{2})^{-(\adimn)/2}(2\pi)^{-(\adimn)/2}\int_{\partial\Omega'}2\langle v,N'(y)\rangle e^{-\frac{\vnorm{y-\rho x}^{2}}{2(1-\rho^{2})}}\,\d y,
\,\forall\,x\in\Sigma.
\end{aligned}
\end{equation}

Then \eqref{gren} says
\begin{equation}\label{grenoneset}
\begin{aligned}
S(\langle v,N\rangle)(x)
&=-\langle v,N'(x)\rangle\frac{1}{\rho}\vnormf{2\overline{\nabla} T_{\rho}(1_{\Omega})(x)},\qquad\forall\,x\in\Sigma'.\\
S'(\langle v,N'\rangle)(x)
&=-\langle v,N(x)\rangle\frac{1}{\rho}\vnormf{2\overline{\nabla} T_{\rho}(1_{\Omega'})(x)},\qquad\forall\,x\in\Sigma.\\
\end{aligned}
\end{equation}

Plugging this into Lemma \ref{lemma7rn}, we get
\begin{equation}\label{four32oneset}
\begin{aligned}
&\frac{1}{2}\frac{\d^{2}}{\d s^{2}}\Big|_{s=0}\sum_{i=1}^{2}\int_{\R^{\adimn}} \int_{\R^{\adimn}} 1_{\Omega_{i}^{(s)}}(y)G(x,y) 1_{\Omega_{i}^{'(s)}}(x)\,\d x\d y\\
&\quad=\int_{\Sigma'} S(f)(x)\cdot f'(x)\gamma_{\adimn}(x) \,\d x
+\int_{\Sigma} S'(f')(x)\cdot f(x)\gamma_{\adimn}(x) \,\d x\\
&\qquad+\int_{\Sigma'}\vnormf{2\overline{\nabla} T_{\rho}(1_{\Omega})(x)}(f'(x))^{2} \gamma_{\adimn}(x)\,\d x
+\int_{\Sigma}\vnormf{2\overline{\nabla} T_{\rho}(1_{\Omega'})(x)}(f(x))^{2} \gamma_{\adimn}(x)\,\d x.
\end{aligned}
\end{equation}
%
%When compared to Lemma \ref{treig2}, Lemma \ref{treig2n} has a negative sign on the right side of the equality, resulting from the positive sign in \eqref{nabeq3n} (as opposed to the negative sign on the right side of \eqref{nabeq3}).  Lemmas \ref{lemma7rn} and \ref{treig2n} then imply the following.
%
%\begin{lemma}[\embolden{Second Variation of Translations, Multiple Sets}]\label{keylemn}
%Let $0<\rho<1$.  Let $v\in\R^{\adimn}$.  Let $\Omega_{1},\ldots,\Omega_{m}$ minimize problem \ref{prob2}.  For each $1\leq i\leq m$, let $\{\Omega_{i}^{(s)}\}_{s\in(-1,1)}$ be the variation of $\Omega_{i}$ corresponding to the constant vector field $X\colonequals v$.  Assume that
%$$\int_{\partial\Omega_{i}}\langle v,N(x)\rangle \gamma_{\adimn}(x)\,\d x=\int_{\partial\Omega_{i}'}\langle v,N(x)\rangle \gamma_{\adimn}(x)\,\d x=0,\qquad\forall\,1\leq i\leq m.$$
%Then
%\begin{flalign*}
%&\frac{\d^{2}}{\d s^{2}}\Big|_{s=0}\sum_{i=1}^{m}\int_{\R^{\adimn}}1_{\Omega_{i}^{(s)}}(x)T_{\rho}1_{\Omega_{i}^{'(s)}}(x)\gamma_{\adimn}(x)\,\d x\\
%&\qquad\qquad\qquad=\Big(-\frac{1}{\rho}+1\Big)\sum_{1\leq i<j\leq m}\int_{\Sigma_{ij}}\vnormf{\overline{\nabla}T_{\rho}(1_{\Omega_{i}'}-1_{\Omega_{j}'})(x)}\langle v,N_{ij}(x)\rangle^{2}\gamma_{\adimn}(x)\,\d x\\
%&\qquad\qquad\qquad\,\,+\Big(-\frac{1}{\rho}+1\Big)\sum_{1\leq i<j\leq m}\int_{\Sigma_{ij}'}\vnormf{\overline{\nabla}T_{\rho}(1_{\Omega_{i}}-1_{\Omega_{j}})(x)}\langle v,N_{ij}'(x)\rangle^{2}\gamma_{\adimn}(x)\,\d x.
%\end{flalign*}
%\end{lemma}
%
%

\begin{definition}
A pair of sets $\Omega,\Omega'\subset\R^{\adimn}$ is called \textbf{locally stable} for noise stability for negative correlation if, for any family of pairs of sets $\{\Omega_{s}\}_{s\in(-1,1)},\{\Omega_{s}'\}_{s\in(-1,1)}$ with $\Omega_{0}=\Omega,\Omega_{0}'=\Omega'$ such that
$$\frac{\d}{\d s}\Big|_{s=0}\gamma_{\adimn}(\Omega_{s})=\frac{\d}{\d s}\Big|_{s=0}\gamma_{\adimn}(\Omega_{s}'),$$
we have
$$\frac{\d^{2}}{\d s^{2}}\Big|_{s=0}\int_{\R^{\adimn}}1_{\Omega^{(s)}}(x)T_{\rho}1_{\Omega^{(s)}}(x)\gamma_{\adimn}(x)\,\d x\leq0.$$

$$\frac{\d^{2}}{\d s^{2}}\Big|_{s=0}
\int_{\R^{\adimn}}1_{\Omega^{(s)}}(x)T_{\rho}1_{\Omega^{'(s)}}(x)\gamma_{\adimn}(x)\,\d x
+\int_{\R^{\adimn}}1_{[\Omega^{(s)}]^{c}}(x)T_{\rho}1_{[\Omega^{'(s)}]^{c}}(x)\gamma_{\adimn}(x)\,\d x\geq0.
$$
\end{definition}

We say two (open) half spaces are opposing if their intersection is either empty or the boundary of this intersection is exactly two parallel hyperplanes.

\begin{cor}\label{rk0}
Opposing half spaces are the only locally stable sets for noise stability for negative correlation.
\end{cor}
An analogous statement for positive correlation was shown in \cite{heilman21}.

\begin{proof}
Suppose without loss of generality that $\Omega\subset\R^{\adimn}$ is not a half space.  Then we have by definition \eqref{sdef} of $S$ the following strict inequality
\begin{equation}\label{dbst}
\int_{\redA'}S(-\abs{\langle v,N\rangle})(x)\cdot \abs{\langle v,N'(x)\rangle} \gamma_{\adimn}(x)\,\d x
<\int_{\redA'}S(\langle v,N\rangle)(x)\cdot\langle v,N(x)\rangle \gamma_{\adimn}(x)\,\d x.
\end{equation}
(If we do not multiply the right by $-1$, the left is positive and the right is negative, so it is necessary to multiply by $-1$ on the left to get the correct inequality.)  In the case that $\Omega=\Omega'\times\R$ (which we can assume by the Dimension Reduction Theorem \ref{mainthm1n}), consider the functions $g(x)\colonequals -x_{\adimn}\abs{\langle v,N(x)\rangle}$ defined on $\redA$ and $g'(x)\colonequals x_{\adimn}\abs{\langle v,N(x)\rangle}$ defined on $\redA'$.  Then
\begin{flalign*}
&\int_{\redA'}\Big(S(-g)(x)+\vnormf{2\overline{\nabla} T_{\rho}1_{\Omega}(x)} g'(x)\Big) g'(x)\gamma_{\adimn}(x)\,\d x\\
&=\int_{\redA'}\Big(\rho S(-\abs{\langle v,N\rangle})(x)\cdot \abs{\langle v,N'(x)\rangle} +\vnormf{2\overline{\nabla} T_{\rho}1_{\Omega}(x)} \abs{\langle v,N'(x)\rangle}\Big)
\abs{\langle v,N'(x)\rangle}\gamma_{\adimn}(x)\,\d x\\
&\stackrel{\eqref{dbst}}{<}\int_{\redA'}\Big(\rho S(\langle v,N\rangle)(x) +\vnormf{2\overline{\nabla} T_{\rho}1_{\Omega}(x)} \langle v,N'(x)\rangle\Big)
\langle v,N'(x)\rangle\gamma_{\adimn}(x)\,\d x
\stackrel{\eqref{grenoneset}}{=}0.
\end{flalign*}
Similarly, we have a (possibly strict) inequality for integrating on $\redA$:
\begin{flalign*}
&\int_{\redA}\Big(S'(g')(x)-\vnormf{2\overline{\nabla} T_{\rho}1_{\Omega}(x)} g(x)\Big) g(x)\gamma_{\adimn}(x)\,\d x\\
&=\int_{\redA'}\Big(\rho S'(-\abs{\langle v,N'\rangle})(x)\cdot \abs{\langle v,N(x)\rangle} +\vnormf{2\overline{\nabla} T_{\rho}1_{\Omega'}(x)} \abs{\langle v,N(x)\rangle}\Big)
\abs{\langle v,N(x)\rangle}\gamma_{\adimn}(x)\,\d x\\
&\geq\int_{\redA}\Big(\rho S'(\langle v,N'\rangle)(x) +\vnormf{2\overline{\nabla} T_{\rho}1_{\Omega'}(x)} \langle v,N(x)\rangle\Big)
\langle v,N(x)\rangle\gamma_{\adimn}(x)\,\d x
\stackrel{\eqref{grenoneset}}{=}0.
\end{flalign*}
So, $\int_{\redA} g(x)\gamma_{\adimn}(x)\,\d x=\int_{\redA'} g'(x)\gamma_{\adimn}(x)\,\d x=0$ while the corresponding variation of $g,g'$ satisfies
$$\frac{\d^{2}}{\d s^{2}}\Big|_{s=0}\sum_{i=1}^{2}\int_{\R^{\adimn}} \int_{\R^{\adimn}} 1_{\Omega_{i}^{(s)}}(y)G(x,y) 1_{\Omega_{i}^{'(s)}}(x)\,\d x\d y<0.$$
That is, the half space is the only stable maximum of noise stability.  More specifically, both $\Omega$ and $\Omega'$ are half spaces.  A two-case comparison shows that the noise stability is larger when one half space contains the other, and smaller when one half space does not contain the other.
\end{proof}

\bibliographystyle{amsalpha}
%\bibliography{12162011}
\def\polhk#1{\setbox0=\hbox{#1}{\ooalign{\hidewidth
  \lower1.5ex\hbox{`}\hidewidth\crcr\unhbox0}}} \def\cprime{$'$}
  \def\cprime{$'$}
\providecommand{\bysame}{\leavevmode\hbox to3em{\hrulefill}\thinspace}
\providecommand{\MR}{\relax\ifhmode\unskip\space\fi MR }
\providecommand{\MRhref}[2]{%
  \href{http://www.ams.org/mathscinet-getitem?mr=#1}{#2}
}
\providecommand{\href}[2]{#2}

\end{document}